\documentclass[10pt, a4paper]{amsart}

\usepackage{amssymb}
\usepackage{amsmath}
\usepackage{amssymb}
\usepackage{amsthm}
\usepackage{bbm}
\usepackage{bm}
\usepackage{mathrsfs}
\usepackage{dsfont}
\usepackage{mathtools}
\usepackage{upref}

\usepackage{enumitem}
\setenumerate{label={\rm (\alph{*})}}
\usepackage[colorlinks=true,urlcolor=blue,linkcolor=blue]{hyperref}

\usepackage[top=3.4cm,bottom=3.4cm,left=3.4cm,right=3.4cm]{geometry}

\newtheorem{theorem}{Theorem}[section]
\newtheorem{lemma}[theorem]{Lemma}
\newtheorem{proposition}[theorem]{Proposition}
\newtheorem{corollary}[theorem]{Corollary}
\newtheorem{definition}[theorem]{Definition}
\newtheorem{remark}[theorem]{Remark}
\newtheorem{example}[theorem]{Example}

\numberwithin{equation}{section}

\normalsize

\DeclareMathOperator*{\essinf}{ess\,inf}
\DeclareMathOperator{\di}{div}

\newcommand{\bd}{\operatorname{BD}}
\newcommand{\bv}{\operatorname{BV}}
\renewcommand{\geq}{\geqslant}

\newcommand{\ep}{\varepsilon}

\newcommand{\dif}{\operatorname{d}\!}

\newcommand{\spt}{\operatorname{spt}}
\newcommand{\A}{A}
\newcommand{\N}{\mathbb{N}}

\newcommand{\R}{\mathds{R}}

\newcommand{\locc}{\operatorname{loc}}

\newcommand{\dista}{\operatorname{dist}}

\newcommand{\wstar}{\stackrel{*}{\rightharpoonup}}
\newcommand{\ball}{\operatorname{B}}

\newcommand{\sobo}{\operatorname{W}}
\newcommand{\lebe}{\operatorname{L}}

\newcommand{\hold}{\operatorname{C}}

\newcommand{\D}{\operatorname{D}\!}

\renewcommand{\leq}{\leqslant}

\newcommand{\id}{\operatorname{Id}}

\newcommand{\sg}{\nabla}
\newcommand{\bmo}{\operatorname{BMO}}

\newcommand{\wbit}{\stackrel{\operatorname{b}}{\rightharpoonup}}

\newcommand{\res}{\mathbin{\vrule height 1.4ex depth 0pt width
0.13ex\vrule height 0.13ex depth 0pt width 1.0ex}}
\begin{document}

\title[On a Neumann Problem for linear growth functionals]{On a Neumann problem for variational functionals of linear growth}

\thanks{L.~Beck is grateful for the support by M.O.P.S.~of the Mathematical Institute at the University of Augsburg. M.~Bul\'{\i}\v{c}ek's work was supported by the ERC-CZ project LL1202 financed by the Ministry of Education, Youth and Sports, Czech Republic. M.~Bul\'{\i}\v{c}ek is a member of the Ne\v{c}as center for Mathematical Modeling. F.~Gmeineder gratefully acknowledges financial support through the EPSRC during his doctoral studies.}

\author[L. Beck]{Lisa Beck}
\address{Institut f\"{u}r Mathematik, Universit\"{a}t Augsburg, Universit\"{a}tsstr. 14, 86159~Augsburg, Germany}
\email{lisa.beck@math.uni-augsburg.de}

\author[M. Bul\'{\i}\v{c}ek]{Miroslav Bul\'{\i}\v{c}ek}
\address{Mathematical Institute, Faculty of Mathematics and Physics, Charles University, Sokolovsk\'{a} 83, 186~75, Prague, Czech Republic}
\email{mbul8060@karlin.mff.cuni.cz}

\author[F. Gmeineder]{Franz Gmeineder}
\address{Mathematisches Institut, Universit\"at Bonn, Endenicher Allee 60, 53115 Bonn, Germany}
\email{fgmeined@math.uni-bonn.de}




\maketitle

\begin{abstract}
We consider a Neumann problem for strictly convex variational functionals of linear growth. We establish the existence of minimisers among $\sobo^{1,1}$-functions provided that the domain under consideration is simply connected. Hence, in this situation, the relaxation of the functional to the space of functions of bounded variation, which has better compactness properties, is not necessary. Similar $\sobo^{1,1}$-regularity results for the corresponding Dirichlet problem are only known under rather restrictive convexity assumptions limiting its non-uniformity up to the borderline case of the minimal surface functional, whereas for the Neumann problem no such quantified version of strong convexity is required.
\end{abstract}

\tableofcontents

\section{Introduction}

Let $\Omega\subset\R^{n}$ be a bounded Lipschitz domain and suppose that $f\in\hold^{1}(\R_{0}^{+})$ is a strictly convex function which satisfies $f(0) = f'(0)= 0$ and which is of linear growth, i.e., there exist two constants $0< \nu \leq L<\infty$ such that
\begin{align}\label{eq:lingrowth}
\nu t - L \leq f(t) \leq L(t+1)\qquad\text{for all } t\in\R^+_0.
\end{align}
Given a map $T_{0}\in \sobo^{1,\infty}(\Omega;\R^{N\times n})$ for some $N\geq 1$, in the present paper we study existence and regularity properties of weak solutions of the system
\begin{align}\label{eq:PDE}
\di\bigg(\frac{f'(|\nabla u|)\nabla u}{|\nabla u|} \bigg)=\di(T_{0})\qquad \text{in }\Omega
\end{align}
subject to the \emph{Neumann-type} boundary condition
\begin{align}\label{eq:neumann}
\frac{f'(|\nabla u|)\nabla u}{|\nabla u|}\cdot\nu_{\partial\Omega}=T_{0}\cdot\nu_{\partial\Omega}   \qquad \text{on }\partial \Omega,
\end{align}
with $\nu_{\partial \Omega}$ denoting the outward pointing unit normal field of the boundary~$\partial \Omega$. In this situation, we have different options to come up with an appropriate concept of weak solutions of~\eqref{eq:PDE} subject to the boundary condition~\eqref{eq:neumann}. Firstly, supposing for the moment that~$u$ belongs to the space~$\hold^{2}(\overline{\Omega};\R^{N})$, we observe that all expressions are well-defined in the classical sense. Thus, applying the inner product to both sides of~\eqref{eq:PDE} with a regular test function $\varphi\in\hold^{1}(\overline{\Omega};\R^{N})$, integrating over $\Omega$ and using the integration by parts formula, we obtain
\begin{multline*}
\int_{\partial\Omega} \frac{f'(|\nabla u|)\nabla u}{|\nabla u |}\cdot \varphi \otimes \nu_{\partial\Omega}  \dif\mathcal{H}^{n-1}  -\int_{\Omega} \frac{f'(|\nabla u|)\nabla u}{|\nabla u |} \cdot \nabla \varphi \dif x \\
= \int_{\partial\Omega} T_{0} \cdot \varphi \otimes \nu_{\partial\Omega} \dif\mathcal{H}^{n-1}-\int_{\Omega} T_{0} \cdot \nabla\varphi \dif x.
\end{multline*}
In view of the Neumann-type constraint~\eqref{eq:neumann}, the boundary terms disappear. Combined with a density argument, this motivates the following definition of a weak solution:

\begin{definition}[Weak Solution]\label{def:weaksolution_Neumann}
Let $T_{0}\in \sobo^{1,\infty}(\Omega;\R^{N\times n})$ and suppose that $f\in\hold^{1}(\R_{0}^{+})$ satisfies $f(0) = f'(0)= 0$ and the linear growth assumption~\eqref{eq:lingrowth}. We say that a function $u\in\sobo^{1,1}(\Omega;\R^{N})$ is a \emph{weak solution} to the system~\eqref{eq:PDE} subject to the Neumann-type boundary constraint~\eqref{eq:neumann} if there holds
\begin{equation}
\label{eq:EL_system}
\int_{\Omega} \frac{f'(|\nabla u|)\nabla u}{|\nabla u |} \cdot \nabla\varphi \dif x
= \int_{\Omega} T_{0} \cdot \nabla\varphi \dif x \qquad \text{for all } \varphi \in \sobo^{1,1}(\Omega;\R^{N}).
\end{equation}
\end{definition}

Alternatively, we may rely on the special structure of the system and interpret it as the Euler--Lagrange system associated to the variational problem
\begin{align}\label{eq:varprin}
\text{to minimise} \quad \mathfrak{F}[w] \coloneqq \int_{\Omega} \big[ f(|\nabla w|)-  T_{0} \cdot \nabla w \big] \dif x \quad \text{among all  } w \in \sobo^{1,1}(\Omega;\R^{N}).
\end{align}
Studying variations of a minimiser in a standard way on the one hand and employing the convexity of the integrand~$f$ on the other hand, we immediately establish the following connection between~\eqref{eq:PDE},~\eqref{eq:neumann} and the variational principle~\eqref{eq:varprin}.

\begin{lemma}\label{lemma:weaksolution_variational}
Let $T_{0}\in \sobo^{1,\infty}(\Omega;\R^{N\times n})$ and suppose that $f\in\hold^{1}(\R_{0}^{+})$ is convex and that it satisfies $f(0) = f'(0)= 0$ and the linear growth assumption~\eqref{eq:lingrowth}. Then a function $u\in\sobo^{1,1}(\Omega;\R^{N})$ is a weak solution of~\eqref{eq:PDE} subject to the Neumann-type boundary constraint~\eqref{eq:neumann} \textup{(}in the sense of Definition~\ref{def:weaksolution_Neumann}\textup{)} if and only if it is a minimiser of the variational problem~\eqref{eq:varprin}.
\end{lemma}

Although we shall exclusively study weak solutions in all of what follows, we wish to mention for the sake of completeness that it is possible to deduce the validity of~\eqref{eq:PDE} subject to~\eqref{eq:neumann} provided that a suitable a~priori regularity assumption on the solution is made:

\begin{lemma}\label{lem:consistencylemma}
Let $T_{0}\in \sobo^{1,\infty}(\Omega;\R^{N\times n})$ and suppose that $f\in\hold^{2}(\R_{0}^{+})$ satisfies $f(0) = f'(0)= 0$. If $u\in\sobo^{2,\infty}(\Omega;\R^{N})$ is a minimiser of the variational principle~\eqref{eq:varprin}, then it satisfies~\eqref{eq:PDE} and~\eqref{eq:neumann} in the pointwise sense.
\end{lemma}

For the reader's convenience, the proof of this lemma is provided in Section~\ref{sec:consistencylemma} of the appendix. Let us further note that the above variational principle~\eqref{eq:varprin} ignores the addition of constants to competitors. To overcome this inherent source of non-uniqueness, we shall additionally require minimisers $u\colon\Omega\to\R^{N}$ to be of vanishing mean value on~$\Omega$, i.e.,~to satisfy
\begin{equation*}
(u)_{\Omega}\coloneqq \frac{1}{\mathscr{L}^{n}(\Omega)}\int_{\Omega}u\dif x = 0.
\end{equation*}

By the linear growth hypothesis~\eqref{eq:lingrowth} and the concomitant lack of weak compactness in the non-reflexive space $\sobo^{1,1}(\Omega;\R^{N})$, minimising sequences for $\mathfrak{F}$ might develop concentrations. Hence, the distributional gradients of minimisers have to be assumed to be matrix-valued Radon measures a priori. This leads to studying a suitably relaxed form of the aforementioned variational problem on the space $\bv(\Omega;\R^{N})$, the space of functions of bounded variation. The purpose of the present paper is to demonstrate that, under some sort of attainability condition imposed on the data~$T_0$, the singular part of the gradients of weak solutions of the system~\eqref{eq:PDE} subject to the Neumann-type constraint~\eqref{eq:neumann} --~or equivalently of minimisers of the variational problem~\eqref{eq:varprin}~-- do in fact vanish, whenever the Lipschitz domain~$\Omega$ is \emph{simply connected}. Thus, in this setting, weak solutions genuinely belong to the space $\sobo^{1,1}(\Omega;\R^{N})$ and the relaxation of the problem to~$\bv(\Omega;\R^{N})$ is indeed not necessary.

Due to the specific form of the variational problem~\eqref{eq:varprin}, this task appears in the spirit of some sort of non-linear potential theory for linear growth problems whose connection to perhaps more familiar settings we shall describe now. The variational problem~\eqref{eq:varprin} formally leads to the Euler--Lagrange system
\begin{align}\label{eq:eulerlag}
\di\bigg(\frac{f'(|\nabla u|)\nabla u}{|\nabla u|}\bigg)=\di(T_{0})\qquad\text{in } \Omega.
\end{align}
Neglecting for a moment the linear growth assumption~\eqref{eq:lingrowth} and setting $f(t) = t^p/p$ for some $p \in (1,\infty)$, the system~\eqref{eq:eulerlag} subject to the boundary condition~\eqref{eq:neumann} corresponds to the weak formulation of the inhomogeneous $p$-Laplacean Neumann problem
\begin{equation*}
\di(|\nabla u|^{p-2}\nabla u)=\di(T_{0})
\end{equation*}
or, equivalently, the minimisation problem~\eqref{eq:varprin}. This, as a consequence of the direct method of the calculus of variations, is solved by some function $u \in \sobo^{1,p}(\Omega;\R^{N})$. Here, a typical issue is to transfer regularity properties of the data~$T_{0}$ to the gradient $\nabla u$ of the solution, or more specifically to the function $\A_{p}(\nabla u)\coloneqq |\nabla u|^{p-2}\nabla u$, which is adapted to the particular growth properties of the elliptic $p$-Laplacean system under consideration. For instance, it is known that $T_{0}\in\lebe^{\infty}(\Omega;\R^{N \times n})$ implies $\A_{p}(\nabla u)\in\bmo_{\locc}(\Omega;\R^{N \times n})$ under some fairly general regularity assumptions on the domain~$\Omega$. This result is optimal in the sense that in general it cannot be improved to $\A_{p}(\nabla u)\in\lebe^{\infty}_{\locc}(\Omega;\R^{N \times n})$: even in the simplest linear case $p=2$ the map $T_{0} \mapsto \nabla u$ is a local singular integral of convolution type which maps $\lebe^{\infty}(\Omega;\R^{N \times n}) \to \bmo_{\locc}(\Omega;\R^{N \times n})$. 

Now, in our situation of~$f$ satisfying the linear growth assumption~\eqref{eq:lingrowth} and setting $\A_f(z)\coloneqq f'(|z|)z/|z|$ for $z\in\R^{N \times n}$, a statement like $\A_f(\nabla u)\in\lebe^{\infty}(\Omega;\R^{N \times n})$ would be vacuous: Since~$f'$ and thus~$\A_f$ is automatically bounded by assumption, we would be able to conclude $\A_f(\D u)\in\lebe^{\infty}(\Omega;\R^{N \times n})$ without further efforts \emph{provided} that~$\D u$ would be known to exist as a function. In this sense, the correct question is under which conditions on $T_{0}$ we can in fact conclude the existence of a $\sobo^{1,1}$-minimiser. As such, the theme of the present paper canonically generalises key aspects of the by now well-known potential theory in the superlinear growth regime (cp.~\cite{IWANIEC83,MINGIONE11,DKS12}) to the linear growth situation. 

Before we embark on a detailed description of our results, we first discuss the main assumption of a suitable coerciveness condition on the functional~$\mathfrak{F}$ which will be imposed throughout the paper.

\subsection{Coerciveness}\label{sec:coerciveness}

Since both constituents of the integrand at our disposal are of linear growth, we must impose an additional balancing condition between~$f$ and $T_{0}$. As a crucial assumption of our paper, we shall therefore require
\begin{align}\label{eq:coercivcond}
\|T_{0}\|_{\lebe^{\infty}(\Omega;\R^{N \times n})}< f^{\infty}(1),
\end{align}
where $f^{\infty}(1)$ is defined as the limit $\lim_{t \to \infty} f(t)/t$. As by convexity of~$f$, the function $t \mapsto f(t)/t$ is non-decreasing, its limit for $t \to \infty$ exists and is in view of~\eqref{eq:lingrowth} indeed finite and strictly positive, with $f^{\infty}(1) \geq f(t)/t$ for all $t \in \R^+$. The significance of this assumption becomes transparent when studying the coerciveness (or its failure) of the functional~$\mathfrak{F}$ in the class $\sobo^{1,1}(\Omega;\R^{N})$ with vanishing mean value on~$\Omega$. In fact, if $T_0 \in \lebe^{\infty}(\Omega;\R^{N \times n})$ satisfies~\eqref{eq:coercivcond}, then we can first determine $R_0$ depending only on~$f$ and $T_0$ such that
\begin{equation*}
 \frac{f(t)}{t} \geq \frac{1}{2} \big(f^{\infty}(1) + \|T_{0}\|_{\lebe^{\infty}(\Omega;\R^{N \times n})}\big) \qquad \text{for } t \geq R_0
\end{equation*}
and then compute, for an arbitrary $w \in \sobo^{1,1}(\Omega;\R^{N})$, that
\begin{align*}
 \mathfrak{F}[w] & = \int_{\Omega} \big[ f(|\nabla w|)-  T_{0} \cdot \nabla w \big] \dif x \\
  & \geq \int_{\Omega \cap \{|\nabla w| \geq R_0 \}} \Big[ \frac{1}{2} \big(f^{\infty}(1) + \|T_{0}\|_{\lebe^{\infty}(\Omega;\R^{N \times n})}\big) - \|T_{0}\|_{\lebe^{\infty}(\Omega;\R^{N \times n})}\Big]  | \nabla w | \dif x \\
  & \qquad {}- \int_{\Omega \cap \{|\nabla w| < R_0 \}} \|T_{0}\|_{\lebe^{\infty}(\Omega;\R^{N \times n})}  | \nabla w | \dif x \\
  & \geq\frac{1}{2} \big(f^{\infty}(1) - \|T_{0}\|_{\lebe^{\infty}(\Omega;\R^{N \times n})}\big) \|\nabla w\|_{\lebe^{1}(\Omega;\R^{N \times n})} -f^{\infty}(1) |\Omega| R_0.
\end{align*}
As a consequence, if $(w_{k})_{k \in \N}$ is a sequence in $\sobo^{1,1}(\Omega;\R^{N})$ with vanishing mean values and $\|w_{k}\|_{\sobo^{1,1}(\Omega;\R^{N})} \to \infty$ as $k\to\infty$, then $\mathfrak{F}[w_{k}]\to \infty$ as $k\to\infty$. Condition~\eqref{eq:coercivcond} thus is an instrumental ingredient to establish the existence of minimisers. To further stress its necessity, we wish to supply the following two examples which demonstrate that, in absence of condition~\eqref{eq:coercivcond}, minimisers do not need to exist at all. This already happens in the scalar case $N=n=1$.

\begin{example}[Non-Existence of minimisers if $\|T_{0}\|_{\lebe^{\infty}(\Omega;\R^{N\times n})}=f^{\infty}(1)$]\label{ex:nonexistence1} We consider the shifted area-integrand
\begin{equation*}
f(t)\coloneqq \sqrt{1+|t|^{2}} - 1 \qquad \text{for } t\in\R
\end{equation*}
\textup{(}which verifies the linear growth assumption~\eqref{eq:lingrowth} with $\nu=L=1$\textup{)}, $T_{0}\equiv 1$ and $\Omega\coloneqq (-1,1)$. In this situation, we have $f^{\infty}(1)=1$ and the functional $\mathfrak{F}$ becomes
\begin{align*}
\mathfrak{F}[w] = \int_{-1}^{1} \big[ \sqrt{1+|w'|^{2}} - 1- w' \big] \dif x ,\qquad \text{for } w\in \sobo^{1,1}((-1,1)).
\end{align*}
Furthermore, since $\sqrt{1+|t|^{2}}\geq t$ for all $t\in\R$, we have $\inf_{\sobo^{1,1}((-1,1))}\mathfrak{F}\geq -2$.

We then define a sequence $(u_k)_{k\in \N}$ of functions in $\sobo^{1,1}((-1,1))$ with vanishing mean value on~$(-1,1)$, by setting $u_{k}(x)\coloneqq kx$ for $k \in \N$. Inserting~$u_{k}$ into $\mathfrak{F}$ yields
\begin{equation*}
\mathfrak{F}[u_{k}]=2 \big[ \sqrt{1+k^{2}}-k - 1 \big] \to - 2 \quad \text{as } k\to\infty
\end{equation*}
so that $\inf_{\sobo^{1,1}((-1,1))}\mathfrak{F}=-2$ indeed. Assuming that a minimiser $u \in\sobo^{1,1}((-1,1))$ of~$\mathfrak{F}$ exists, we deduce, by positivity of the integrand, that $1+|v'|^{2}=|v'|^{2}$ holds $\mathscr{L}^{1}$-a.e., a contradiction. Therefore, no minimiser of $\mathfrak{F}$ exists in $\sobo^{1,1}((-1,1))$.
\end{example}

\begin{example}[Unboundedness of $\mathfrak{F}$ from below if $\|T_{0}\|_{\lebe^{\infty}(\Omega;\R^{N\times n})} > f^{\infty}(1)$]\label{ex:nonexistence2} In the setting of the previous example, we consider $T_{0}\equiv c$ for a constant $c>1$. For the same choice of the sequence $(u_{k})_{k \in \N}$, we then obtain
\begin{align*}
\mathfrak{F}[u_{k}]=2 \big[ \sqrt{1+k^{2}}-ck - 1 \big] \to -\infty,\qquad \text{as } k\to\infty
\end{align*}
which in conclusion shows $\inf_{\sobo^{1,1}((-1,1))}\mathfrak{F}=-\infty$.
\end{example}

In principle, the reasoning employed in Example~\ref{ex:nonexistence1} does not genuinely rule out the non-existence of minimisers for the so-called \emph{relaxed problem}, i.e., the minimisation of a suitable extension of $\mathfrak{F}$ to the space $\bv(\Omega;\R^{N\times n})$. However, even for the relaxed problem the assumption~\eqref{eq:coercivcond} turns out to be necessary for generalised minimisers to exist, see Example~\ref{ex:nonexistence3}.

\begin{remark}
Under the assumption~\eqref{eq:coercivcond} we can rewrite $T_{0}\in\sobo^{1,\infty}(\Omega;\R^{N\times n})$ as
\begin{equation*}
 T_{0} = \frac{f'(|S_{0}|)S_{0}}{|S_{0}|} \quad \text{for } S_0 \text{ given as } S_0 \coloneqq \frac{T_0}{|T_0|} (f')^{-1}(|T_0|).
\end{equation*}
Since~$f'$ is strictly increasing with values in~$[0,f^\infty(1))$ \textup{(}thus, invertible on this set\textup{)}, the map~$S_0$ is well-defined. With this identification, assumption~\eqref{eq:coercivcond} guarantees that~$\di T_0$ on the left-hand side of the system~\eqref{eq:PDE} is of the same structure as its left-hand side involving the unknown and thus, in principle, can be attained.
\end{remark}

\subsection{Main Result and Discussion}

We now pass to the description of the main result of the present paper. As mentioned above, one can easily extend the functionals $\mathfrak{F}$ to the space $\bv(\Omega;\R^{N})$. This will be done in a slightly more general setup than for functionals with radially symmetric integrands, and by means of the direct method of the calculus of variations, existence of \emph{$\bv$- \textup{(}or generalised\textup{)}} minimisers then follows (for the precise statement the reader is referred to Proposition~\ref{prop:existence_relaxed}). However, the main result of the present paper is the existence of $\sobo^{1,1}$-minimisers for $\mathfrak{F}$ in the radially symmetric case provided that~$\Omega$ is simply connected. More precisely, we will establish the following

\begin{theorem}\label{thm:main}
Let $\Omega$ be a simply connected, bounded Lipschitz domain in $\R^{n}$. Consider a strictly convex function $f\in\hold^{2}(\R_{0}^{+})$ which satisfies $f(0) = f'(0)= 0$, the linear growth condition~\eqref{eq:lingrowth} and the bound
\begin{equation}
\label{eq:sec_der_bound}
 f''(t) \leq L (1+t)^{-1} \qquad\text{for all } t \in\R^+_0,
\end{equation}
and let $T_{0}\in \sobo^{2,\infty}(\Omega;\R^{N\times n})$ verify~\eqref{eq:coercivcond}. Then there exists a weak solution $u\in\sobo^{1,1}(\Omega;\R^{N})$ of the system~\eqref{eq:PDE} subject to the Neumann-type boundary constraint~\eqref{eq:neumann} in the sense of Definition~\ref{def:weaksolution_Neumann}, and this weak solution is unique within the class of all admissible competitor maps $v\in\sobo^{1,1}(\Omega;\R^{N})$ that satisfy $(v)_{\Omega}=0$.
\end{theorem}

Let us comment on our theorem, its strategy of proof and related results from the literature. To the best of our knowledge, Theorem~\ref{thm:main} is \emph{the first $\sobo^{1,1}$-regularity result for a minimisation problem involving a linear growth condition on the integrand without requiring a quantified version of strong convexity}, even though the result applies only to the Neumann problem and not to the Dirichlet problem. In order to compare the outcome of Theorem~\ref{thm:main} with the available results, let us report on the relevant regularity results in the literature for the Dirichlet problem. This (again with a radially symmetric integrand) is just the variational problem
\begin{equation*}
\text{to minimise} \quad \int_{\Omega}f(|\nabla w|)\dif x \quad \text{over } w \in u_0 +\sobo^{1,1}_{0}(\Omega;\R^{N}).
\end{equation*}
subject to some prescribed boundary values $u_0 \in \sobo^{1,1}(\Omega;\R^N)$. Due to the lack of weak compactness of norm-bounded sequences in the space $u_0 + \sobo^{1,1}_0(\Omega;\R^N)$, one equally passes to the relaxed formulation and is thereby lead to the concept of $\bv$-minimisers. For the latter, its measure derivative may be non-trivial in the interior and, on the other hand, the prescribed boundary values might not be attained. The phenomenon of non-attainment of prescribed boundary values is well-known to occur already for minimal surfaces, while interior singularities can be ruled out in certain instances. In this regard, we briefly recall the notion of \emph{$\mu$-ellipticity} which quantifies the degeneration of second order derivatives of $z \mapsto f(|z|)$ and therefore represents an instrumental ingredient for deriving higher regularity for $\bv$-minimisers. We say that~$f$ is \emph{$\mu$-elliptic} for some $\mu \in (1,\infty)$ if
\begin{equation*}
\nu (1+|z|^{2})^{-\frac{\mu}{2}} |\zeta|^{2} \leq \D_{zz} f(|z|) [\zeta , \zeta]
\end{equation*}
holds for all $z,\zeta \in\R^{N \times n}$ (after possibly choosing the constant $\nu > 0$ from the growth condition~\eqref{eq:lingrowth} smaller). The impact of $\mu$-ellipticity on the regularity of generalised minimisers has been investigated to considerable detail by Bildhauer and Fuchs~\cite{BILFUC02b,BILDHAUER02,BILDHAUER03,BILDHAUERBUCH} (and by Fuchs and Mingione~\cite{FUCMIN00} for nearly linear growth problems). More specifically, under the mild degeneration condition $\mu \in (1,3)$, minimisers are in fact $\hold^{1}_{\locc}$-regular (see \cite[Theorem~2.7]{BILDHAUER02}, but also \cite[Theorem~B]{MARPAPI06} and \cite[Theorem~1.3]{BECSCH15}), while in the limit case with degeneration $\mu = 3$ (as for the area functional) the minimisers are still $\sobo^{1,1}$-regular (see \cite[Theorem~2.5]{BILDHAUER02} and \cite[Corollary~1.13]{BECSCH13}). The method of proof for these results consists in establishing uniform higher integrability of the gradients of suitable minimising sequences, which then is conserved in the passage to the limit. This seems to require the bound $\mu \leq 3$, and in fact, it is not known whether $\sobo^{1,1}$-regularity still holds or whether interior singularities might arise for $\mu > 3$. In this situation, however, one still has partial (H\"{o}lder) regularity results (cf.~\cite{ANZGIA88,SCHMIDT14,GMEKRI19} for some results in this direction), while  a counterexample of a minimiser in $\bv \setminus \sobo^{1,1}(\Omega)$ was constructed, so far, only for the non-autonomous case (see \cite[Theorem~4.39]{BILDHAUERBUCH}, building on a one-dimensional example from~\cite{GIAMODSOU79}).

In fact, the analysis of the Neumann problem is often omitted in the literature since the methods used for the Dirichlet problem can, as far as such interior estimates are concerned, be easily adapted also to our setting with the presence of~$T_0$. This is for example the case in the result of Temam~\cite{TEMAM71} (see also~\cite[Chapter~V.4]{EKETEM99}), where the existence of a (scalar-valued) $\sobo^{1,1}(\Omega)$-solution is shown for the Neumann problem, when dealing with functionals of linear growth and with degeneration not worse than for the minimal surface equation. However, let us emphasize that we here \emph{go beyond what is known for the Dirichlet problem} by showing that every $\bv$-minimiser belongs to $\sobo^{1,1}(\Omega;\R^{N})$ for \emph{all} strictly convex integrands regardless of any $\mu$-convexity assumption. In particular, the result holds for the prototypical integrands
\begin{equation*}
 f(t) \coloneqq \int_0^t (1+\tau^{\mu-1})^{-1/(\mu-1)} \tau \dif \tau
\end{equation*}
(satisfying the $\mu$-ellipticity condition) with any $\mu \in (1,\infty)$, but also for more general ones.

\begin{remark}
\label{rem_h_monotone}
In this context, let us note that under the assumptions of Theorem~\ref{thm:main} on the function~$f$, we can in general still ensure the existence of a continuous function $h\colon \R_{0}^{+}\to\R^{+}$ fulfilling $h>0$ a.e. in $\R_{0}^{+}$ such that
\begin{align}\label{eq:hmonotone}
h(|z|)|\zeta|^{2}\leq  \D_{zz} f(|z|) [\zeta , \zeta] \le L  \frac{|\zeta|^2}{1+|z|}
\end{align}
holds for all $z,\zeta \in\R^{N \times n}$ \textup{(}see Section~\ref{sec:hconvex} for a short proof\textup{)}. This notion of \emph{$h$-monotonicity} is a generalisation of the aforementioned $\mu$-ellipticity and reduces to that for the particular choice $h(t)\coloneqq (1+|t|)^{-\mu}$.
\end{remark}

We now comment briefly on the strategy of proof. In a first step and as it is usually done also for the Dirichlet problem (as for example in \cite{BILFUC02b,BILDHAUER02,BILDHAUER03,BECSCH15} mentioned above), we employ a classical vanishing viscosity approach. This yields specific minimising sequences satisfying good a~priori estimates. However, we then do \emph{not} use techniques designed to obtain higher integrability of the gradients of the solutions to these approximate problems. Instead, building on a strategy developed in~\cite{BECBULMALSUE17}, we prove that the relevant minimising sequences converge $\mathscr{L}^{n}$-a.e.~to an $\lebe^{1}$-map, which is then shown to be curl-free in the sense of distributions. It is only at this stage that we need the condition on $\Omega$ to be simply connected, which is sufficient to deduce that the aforementioned limit is actually the gradient of a $\sobo^{1,1}(\Omega;\R^{N})$-map~$u$. Now, by the pointwise convergence of the gradients, we finally obtain that this~$u$ is in fact a minimiser for the variational problem~\eqref{eq:varprin}. Unfortunately, this final step of the verification of the minimality property seems to fail for the Dirichlet problem. Here, the essential obstruction is that the boundary values of the minimising sequence are not controlled when only pointwise convergence of the gradients is available. Moreover, it would also be interesting to know whether the assumption on~$\Omega$ to be simply connected is mandatory in Theorem~\ref{thm:main}.

With the existence result of Theorem~\ref{thm:main} at hand, we can now return to our initial potential theoretic question of the regularity of~$A_{f}(\nabla u)$. Under the same assumptions as in Theorem~\ref{thm:main}, some regularity of~$T_0$ is inherited and we indeed obtain $A_{f}(\nabla u)\in\sobo_{\locc}^{1,2}(\Omega;\R^{N\times n})$, see Theorem~\ref{thm_regularity_dual}. We further note that in this situation the quantity $A_{f}(\nabla u)-T_{0}$ takes actually the role of the dual solution (in the sense of convex duality, cp.~Section~\ref{sec:duality}, and see~\cite{EKETEM99,FUCSER99} for related relevant contributions in the superlinear growth case), while in more general situations this Sobolev regularity for the dual solution still survives (even though it cannot necessarily be represented as $A_{f}(\nabla u)-T_{0}$ by the possible presence of the singular part in~$\D u$).

\subsection{Organisation of the Paper}

To conclude the introduction, we give a short outline of the paper. In Section~\ref{sec:prelims} we gather some preliminary results needed later on, in particular, we remind Chacon's biting lemma and state a suitable Sobolev-type version of the classical Poincar\'e lemma, which allows us to recover the gradient structure, whenever an $\lebe^1$-function is curl-free in the sense of distributions on a simply connected domain. In Section~\ref{sec:main} we then establish Theorem~\ref{thm:main} in several steps as already sketched in detail above. In Section~\ref{sec:relaxedanddual} we explain the relaxed primal problem, i.e., the extension of the functional~$\mathfrak{F}$ originally defined only on the space $\sobo^{1,1}(\Omega;\R^N)$ to the larger space $\bv(\Omega;\R^N)$ possessing better compactness properties, and the notion of generalised minimisers. Their existence is then proved, and this is in particular of interest in the case of non-simply connected domains, where we cannot ensure the existence of a $\sobo^{1,1}$-minimiser via Theorem~\ref{thm:main}. In this section we further discuss an alternative approach to the minimisation problem~\eqref{eq:varprin}, namely its dual problem in the sense of convex analysis. In particular, we here identify the correct setup and then link the dual formulation to the primal (relaxed) one in a precise manner. In Section~\ref{sec:appendix} we finally collect some supplementary material for the convenience of the reader.

\section{Preliminaries}\label{sec:prelims}

\subsection{General Notation}
Throughout the paper, $\Omega$ is a simply connected, bounded Lip\-schitz domain in $\R^{n}$. Given $x\in\R^{n}$ and $r>0$, we denote by $\ball(x,r)\coloneqq \{y\in\R^{n}\colon |x-y|<r\}$ the open ball with radius $r>0$ centered at~$x \in \R^n$. For the unit-sphere $\{ x \in \R^k \colon |x| = 1\}$ we further write $\mathbb{S}^{k-1}$. Given $a\in\R^{N}$ and $b\in\R^{n}$, we denote by $a\otimes b\coloneqq ab^{\mathsf{T}}\in\R^{N\times n}$ the tensor product of~$a$ and~$b$. Given a bounded set~$U$ in~$\R^{n}$, we denote by $\mathcal{M}(U;\R^{m})$ the $\R^{m}$-valued Radon measures on~$U$ of finite total variation and denote the space of all bounded continuous functions $U\to\R^{m}$ by $\hold_{b}(U;\R^{m})$. Finally, we denote by $\mu\res A$ the restriction of $\mu$ to a Borel set~$A$ of~$U$, i.e., $(\mu\res A)(V)\coloneqq \mu(A\cap V)$ for Borel sets $V\subset U$.

\subsection{On the gradient structure}

In this section we collect auxiliary estimates and background results that will be useful in the proof of our main result below, when identifying an $\lebe^1$-function with the gradient of a $\sobo^{1,1}$-function. We begin with recording the following version of Chacon's biting lemma:

\begin{lemma}[Chacon's biting lemma, \cite{BALMUR89}]
\label{lemma_biting}
Let $(E_{k})_{k \in \N}$ be a bounded sequence in $\lebe^{1}(\Omega;\R^{m})$. Then there exist a subsequence $(E_{k(\ell)})_{\ell \in \N}$ and a function $E \in \lebe^{1}(\Omega;\R^{m})$ such that $(E_{k(\ell)})_{\ell \in \N}$ converges weakly to~$E$ in the biting sense in $\lebe^{1}(\Omega;\R^{m})$, that is, there exists an increasing sequence $(\Omega_{j})_{j \in \N}$ of measurable sets contained in~$\Omega$ with $\mathscr{L}^{n}(\Omega\setminus\Omega_{j})\to 0$ such that
\begin{align*}
E_{k(\ell)}\rightharpoonup E\qquad\text{weakly in } \lebe^{1}(\Omega_{j};\R^{m})\;\text{as } \ell\to\infty
\end{align*}
for every fixed $j\in\mathbb{N}$.
\end{lemma}

We shall apply Chacon's biting lemma to the gradients of a minimising sequence of the functional~$\mathfrak{F}$ in $\sobo^{1,1}(\Omega;\R^{N})$, hence, to gradients of functions in $\sobo^{1,1}(\Omega;\R^{N})$. In order to deduce a gradient structure of the limit, we will show in the first step, that the limit is curl-free in the sense of distributions, according to the following

\begin{definition}
\label{def_curl_free}
We call a function $e \in \hold^{1}(\Omega;\R^n)$ \emph{curl-free} if for all $i, j \in \{1,\ldots,n\}$ there holds
\begin{equation*}
 \partial_j e_i - \partial_i e_j = 0.
\end{equation*}
Similarly, we call a function $e \in \lebe^{1}(\Omega;\R^{n})$ \emph{curl-free in the sense of distributions} if for any $\varphi \in \hold^{1}_0(\Omega)$ and all $i, j \in \{1,\ldots,n\}$ there holds
\begin{equation*}
 \int_\Omega (e \otimes \nabla)_{ij} \varphi \dif x \coloneqq  \int_\Omega \big( e_i \partial_j \varphi - e_j \partial_i \varphi) \dif x = 0.
\end{equation*}
\end{definition}

\begin{remark}
\label{rem_curl}
\strut
\begin{enumerate}[font=\normalfont, label=(\roman{*}), ref=(\roman{*})]
 \item In order to verify the curl-free condition, one only needs to check the condition for all indices $i<j$, hence, we have $n(n-1)/2$ conditions in total. In particular, for $n=2$, the curl is defined as a scalar function, while for $n=3$ as a $3$-dimensional vectorial function.
 \item If $e \in \lebe^{p}(\Omega;\R^{n})$ for some $p \in [1,\infty]$, then we can take by approximation test functions $\varphi \in \sobo^{1,q}_0(\Omega)$ for $q \in [1,\infty]$ such that $1/p + 1/q=1$. In this case, we find
 \begin{equation*}
   \int_\Omega (e \otimes \nabla) \varphi \dif x \leq C(n) \|e\|_{\lebe^{p}(\Omega;\R^{n})} \|\nabla \varphi \|_{\lebe^{q}(\Omega)}.
 \end{equation*}
\end{enumerate}
\end{remark}

If $e = \nabla w$ for some function $w \in \sobo^{2,1}(\Omega)$, then~$e$ is obviously curl-free via the integration by parts formula. However, the gradient structure is not only sufficient, but indeed necessary for the curl-free condition if~$\Omega$ is a simply connected domain. The precise statement of this Sobolev-type version of the usual Poincar\'{e} Lemma is as follows:

\begin{lemma}\label{lem:Poinclemma}
Let $\Omega\subset\R^{n}$ be a simply connected bounded Lipschitz domain. If a function $E \in\lebe^{1}(\Omega;\R^{N\times n})$ is curl-free in the sense of distributions on $\Omega$, then there exists a function $v\in\sobo^{1,1}(\Omega;\R^{N})$ such that $\nabla v = E$ holds $\mathscr{L}^{n}$-a.e.~in~$\Omega$.
\end{lemma}

\begin{proof}
We first note that the statement is clear if $E\in\hold^1(\Omega;\R^{N\times n})$ is curl-free in the classical sense. Indeed, in this case, we associate to~$E$ the $1$-forms $\omega^\alpha \coloneqq  E^\alpha_1 \dif x_1 + \ldots + E^\alpha_n \dif x_n$, for $\alpha \in \{1,\ldots,N\}$, and we observe that the curl-free condition simply means that each $\omega^\alpha$ is closed. By means of the classical Poincar\'{e} lemma, see e.g.~\cite{SPIVAK79}, it is therefore exact, i.e., we find $0$-forms $v^\alpha$ with $\omega^\alpha = \dif v^\alpha$, for each $\alpha \in \{1,\ldots,N\}$, which precisely means $\nabla v = E$ in~$\Omega$.

The assertion of the lemma now follows by approximation. To this end, let $K\Subset\Omega$ be a simply connected open set. Given $0<\varepsilon<\tfrac{1}{2}\dista(K,\partial\Omega)$, the mollifications $E_{\varepsilon}\colon K +\ball(0,\varepsilon) \to \R^{N\times n}$, defined by convolution $E_{\varepsilon}\coloneqq \rho_{\varepsilon} * E$ with a standard mollifying kernel~$\rho_\varepsilon(x) \coloneqq \varepsilon^{-n} \rho(x/\varepsilon)$ for some non-negative, rotationally symmetric function~$\rho \in \hold_{c}^{\infty}(\ball(0,1))$ with $\| \rho\|_{L^1(\ball(0,1))}=1$, are well-defined and smooth. Furthermore, for every test function $\psi\in\hold_{c}^{\infty}(K;\R^{N\times n})$ we get via Fubini's theorem the relation
\begin{equation*}
\int_\Omega (\rho_\varepsilon * E^\alpha_i) \partial_j \psi \dif x = \int_\Omega E^\alpha_i  \partial_j (\rho_\varepsilon *  \psi) \dif x
\end{equation*}
for all $i,j \in \{1,\ldots,n\}$ and $\alpha \in \{ 1,\ldots,N\}$. As a consequence, $E_{\varepsilon}$ is curl-free in the sense of distributions on~$K$, and thus, by the fundamental theorem of calculus, also in the classical sense. Therefore, by the classical Poincar\'{e} lemma  mentioned above, we find a function $v_{\varepsilon}\in\hold^{1}(K;\R^{N})$ with $\nabla v_{\varepsilon}=E_{\varepsilon}$ on~$K$, and we may also suppose $(v_\varepsilon)_K=0$. With the strong convergence $E_{\varepsilon}\to E$ in $\lebe^{1}(K;\R^{N\times n})$ as $\varepsilon\searrow 0$ by the usual properties of mollifications and with the Poincar\'{e inequality}, we see that $(v_{\varepsilon})_{\varepsilon}$ is a Cauchy sequence in $\sobo^{1,1}(K;\R^{N})$ and hence converges strongly in $\sobo^{1,1}(K;\R^{N})$ to a limit $v_K \in\sobo^{1,1}(K;\R^{N})$. In order to identify $\nabla v_K = E$ a.e.~on~$K$ we calculate, for arbitrary $\varphi\in\hold_{c}^{\infty}(K;\R^{N\times n})$,
\begin{align*}
\left\vert\int_{\Omega} (\nabla v_K - E) \cdot \varphi \dif x\right\vert & = \lim_{\varepsilon\searrow 0}\left\vert\int_{\Omega} (\nabla v_{\varepsilon}-E) \cdot \varphi \dif x \right\vert \\ & \leq \|\varphi\|_{\lebe^{\infty}(K;\R^{N\times n})}\lim_{\varepsilon \searrow 0}\|E_{\varepsilon}-E\|_{\lebe^{1}(K;\R^{N\times n})}=0.
\end{align*}
It only remains to justify that we find a function $v \in \sobo^{1,1}(\Omega;\R^{N\times n})$ such that $\nabla v = E$ holds $\mathscr{L}^{n}$-a.e.~on all of~$\Omega$. To this end, we notice that the sets $\Omega_{\delta}\coloneqq \{x \in \Omega \colon\dista(x,\partial\Omega)>\delta\}$ are simply connected Lipschitz domains provided that $\delta \in (0,\delta_0)$ for some sufficiently small~$\delta_0 >0$, with $\Omega_{\delta}\nearrow\Omega$ as $\delta \searrow 0$. Furthermore, we fix $\delta_1 \leq \delta_0$ such that $2 \mathscr{L}^{n}(\Omega_{\delta_1}) \geq \mathscr{L}^{n}(\Omega)$. With the previous arguments we then find, for every $\delta \leq \delta_1$, a function $v_\delta \in \sobo^{1,1}(\Omega;\R^{N})$ (extended via the extension operator in $\Omega \setminus \Omega_\delta$) such that $\nabla v_\delta = E$ holds a.e.~in~$\Omega_\delta$, and we may further suppose $(v_\delta)_{\Omega_{\delta_1}}=0$. It is easy to see that $(v_\delta)_{\delta \in (0,\delta_1)}$ is a Cauchy family in $\sobo^{1,1}(\Omega;\R^{N})$, with a limit function $v \in \sobo^{1,1}(\Omega;\R^{N})$. Arguing via the pointwise convergence of $(\nabla v_\delta)_{\delta \in (0,\delta_1)}$ for a subsequence (or, alternatively, via the fundamental theorem of calculus as before) we finally end up with the fact that $\nabla v = E$ holds $\mathscr{L}^{n}$-a.e.~in~$\Omega$, which completes the proof.
\end{proof}

\section{Proof of the Main Theorem}\label{sec:main}

\subsection{Existence of solutions for approximate problems}\label{sec:approximation}
Aiming for the existence of a weak solution of the system~\eqref{eq:PDE} subject to the Neumann-type boundary constraint~\eqref{eq:neumann}, or equivalently of a minimiser for the variational principle~\eqref{eq:varprin}, in the class $\sobo^{1,1}(\Omega;\R^N)$, we start to investigate in this section boundedness and convergence properties of a suitable approximating sequence. This sequence, in turn, is obtained by means of a vanishing viscosity-type approach, meaning that on the level of the elliptic system~\eqref{eq:PDE} we add a Laplacean to the differential operator, or on the level of the functional we add the Dirichlet energy (both with small prefactor) to the functional~$\mathfrak{F}$. As a consequence, we can work in these approximations with solutions of class $\sobo^{1,2}(\Omega;\R^N)$. It is easy to see that all arguments which are outlined in this section for the functional~$\mathfrak{F}$ with radially symmetric integrands~$f$ do in fact also apply to more general functionals (as described in~\eqref{eq:varprin_general}
later on) without the radial structure. However, it is in the subsequent sections when we need to rely on the Uhlenbeck structure of the integrands~$f$, in order to obtain the $\sobo^{1,1}$-regularity as claimed in Theorem~\ref{thm:main}.

Let us now introduce, in an intermediate step, the approximate functionals
\begin{align}
\label{eq:approx1}
\begin{split} \mathfrak{F}_{k}[w]\coloneqq \mathfrak{F}[w]+ (2k)^{-1} \int_{\Omega} |\nabla w|^{2} \dif x \coloneqq \int_{\Omega}f_{k}(|\nabla w|)\dif x -\int_{\Omega} T_{0} \cdot \nabla w \dif x
\end{split}
\end{align}
for functions $w \in \sobo^{1,2}(\Omega;\R^{N})$ and all $k \in \N$, where we have set $f_{k}(t)\coloneqq f(t)+ (2k)^{-1} t^2$ for $t \in \R^+_0$. In the first step we establish the existence of a sequence of functions $(u_k)_{k \in \N}$ in $\sobo^{1,2}(\Omega;\R^{N})$ such that, for each $k \in \N$, the function~$u_k$ has vanishing mean value $(u_{k})_{\Omega}=0$ on~$\Omega$ and minimises the functional~$\mathfrak{F}_{k}$ among all functions in $\sobo^{1,2}(\Omega;\R^{N})$.

\begin{lemma}
\label{lemma_existence_uk}
Consider a convex function $f\in\hold^{1}(\R_{0}^{+})$ satisfying $f(0) = f'(0)= 0$ and the linear growth condition~\eqref{eq:lingrowth}, and let $T_{0}\in \lebe^{\infty}(\Omega;\R^{N\times n})$ verify~\eqref{eq:coercivcond}. Then, for every $k \in \N$, the functional $\mathfrak{F}_{k}$ defined in~\eqref{eq:approx1} admits a \textup{(}unique\textup{)} minimiser $u_{k} \in \sobo^{1,2}(\Omega;\R^{N})$ satisfying $(u_{k})_{\Omega}=0$ and
\begin{equation}
\label{eq:norm_bound_u_k}
 \|\nabla u_k \|_{\lebe^{1}(\Omega;\R^{N \times n})} + k^{-1} \|\nabla u_{k} \|_{\lebe^{2}(\Omega;\R^{N \times n})}^2 \leq C \big(1 + \mathfrak{F}_{k}[u_{k}] \big)
\end{equation}
for a constant~$C$ depending only on $\Omega$,~$f$ and~$\|T_{0}\|_{\lebe^{\infty}(\Omega;\R^{N \times n})}$.
\end{lemma}

\begin{proof}
The existence of the minimiser~$u_k$ is a consequence of the direct method of the calculus of variations, for each fixed $k \in \N$. In fact, due to assumption~\eqref{eq:coercivcond} on~$T_0$ (implying coerciveness, cp.~Section~\ref{sec:coerciveness}), the functional~$\mathfrak{F}$ and thus also each of the functionals~$\mathfrak{F}_{k}$ is bounded from below via
\begin{equation*}
 \gamma  \|\nabla w\|_{\lebe^{1}(\Omega;\R^{N \times n})} +  (2k)^{-1}  \|\nabla w\|_{\lebe^{2}(\Omega;\R^{N \times n})}^2 \leq \mathfrak{F}_{k}[w] + C \mathscr{L}^{n}(\Omega)
\end{equation*}
for all functions $w \in\sobo^{1,2}(\Omega;\R^{N})$ and $k \in \N$, with constants~$\gamma$ and~$C$ depending only on~$f$ and~$\|T_{0}\|_{\lebe^{\infty}(\Omega;\R^{N \times n})}$. As a consequence, via Poincar\'{e}'s inequality in the zero-mean version, we find that every minimising sequence $(w_{k,\ell})_{\ell \in \N}$ of~$\mathfrak{F}_{k}$ in the set $\mathcal{C} \coloneqq  \{w\in\sobo^{1,2}(\Omega;\R^{N}) \colon (w)_{\Omega}=0\}$, i.e.,~which satisfies $\mathfrak{F}_{k}(w_{k,\ell})\to \inf_{\mathcal{C}}\mathfrak{F}_{k}$ as $\ell \to\infty$, is bounded in $\sobo^{1,2}(\Omega;\R^{N})$. Since the latter space is reflexive, the classical Banach--Alaoglu Theorem gives a non-relabeled subsequence and a limit map $u_k \in \sobo^{1,2}(\Omega;\R^{N})$ such that $w_{k,\ell}\rightharpoonup u_k$ as $\ell\to\infty$, $(u_{k})_{\Omega}=0$ and the estimate~\eqref{eq:norm_bound_u_k} are satisfied. Now, since by convexity of its integrand the functional~$\mathfrak{F}_{k}$ is lower semi-continuous with respect to weak
convergence in $\sobo^{1,2}(\Omega;\R^{N})$, cp. \cite[Theorem 3.23]{DACOROGNA08}, we obtain $\mathfrak{F}_{k}[u_k]\leq \liminf_{\ell\to\infty} \mathfrak{F}_{k}[w_{k,\ell}]$ for each $k \in \N$. Thus, taking advantage of the strict convexity of the integrand of~$\mathfrak{F}_{k}$, we have shown that~$u_{k}$ is indeed the unique minimiser of $\mathfrak{F}_{k}$ in $\mathcal{C}$, and the proof of the lemma is complete.
\end{proof}

Once the existence of minimisers is ensured, we note that every minimiser $u_k \in \sobo^{1,2}(\Omega;\R^{N})$ of the functional~$\mathfrak{F}_{k}$ in $\sobo^{1,2}(\Omega;\R^{N})$ also satisfies the Euler--Lagrange system
\begin{align}\label{eq:approxel}
 \int_{\Omega} \A_{k}(\sg u_{k}) \cdot \sg \varphi \dif x = \int_{\Omega} T_{0} \cdot \sg \varphi \dif x
\end{align}
for all functions $\varphi \in W^{1,2}(\Omega;\R^{N})$, where the regularised tensor functions $\A_{k} \colon \R^{N\times n}\to\R^{N\times n}$, for $k \in \N$, are given by
\begin{align}\label{eq:Aregularised}
\A_{k}(z)\coloneqq \A(z)+ k^{-1} z \coloneqq f'(|z|) \frac{z}{|z|} + k^{-1} z ,\qquad \text{for all } z\in\R^{N\times n}.
\end{align}
Indeed,~\eqref{eq:approxel} is a simple consequence of the facts that the function $u_k + t\varphi \in \sobo^{1,2}(\Omega;\R^{N})$ is an admissible competitor for each $t \in \R$ and that $t \mapsto \mathfrak{F}_{k}[u_k + t\varphi]$ attains its minimum for $t=0$ (cp.~also Lemma~\ref{lemma:weaksolution_variational}). Let us further recall that, as a consequence of the convexity of~$f$ with $f(0) = f'(0)= 0$, the linear growth condition~\eqref{eq:lingrowth} and the upper bound~\eqref{eq:sec_der_bound} of~$f''$, we can work with the growth conditions
\begin{equation}
\label{eq:growth_D_z_A}
 h(z)|\zeta|^{2}\leq \D_z \A(z) [\zeta , \zeta] = \D_{zz} f(|z|) [\zeta , \zeta] \leq 2L \frac{|\zeta|^{2}}{1+|z|}
\end{equation}
for all $z,\zeta \in \R^{N \times n}$, where~$h$ is the function introduced in Remark~\ref{rem_h_monotone}.

Similarly as in \cite[Lemmata~3.2 and~3.3]{BECSCH15}, we next show that the functional~$\mathfrak{F}_{k}$ is indeed an approximation of the original functional~$\mathfrak{F}$ with respect to minimisation in $\sobo^{1,1}(\Omega;\R^{N})$, in the sense that the minimisers~$u_{k}$ of~$\mathfrak{F}_{k}$ form a minimising sequence for~$\mathfrak{F}$ in $\sobo^{1,1}(\Omega;\R^{N})$. Moreover, we infer a first uniform bound for the sequence $(u_{k})_{k \in \N}$.

\begin{corollary}
\label{cor_uniform_bound_uk}
Consider a convex function $f\in\hold^{1}(\R_{0}^{+})$ satisfying $f(0) = f'(0)= 0$ and the linear growth condition~\eqref{eq:lingrowth}, and let $T_{0}\in \lebe^{\infty}(\Omega;\R^{N\times n})$ verify~\eqref{eq:coercivcond}. Then the sequence $(u_{k})_{k \in \N}$ of minimisers~$u_k$ of the functionals~$\mathfrak{F}_{k}$ from Lemma~\eqref{lemma_existence_uk} is a minimising sequence for~$\mathfrak{F}$ in $\sobo^{1,1}(\Omega;\R^{N})$ with
\begin{equation*}
  \lim_{k \to \infty} \inf_{\sobo^{1,2}(\Omega;\R^{N})} \mathfrak{F}_{k} =  \inf_{\sobo^{1,1}(\Omega;\R^{N})} \mathfrak{F}.
\end{equation*}
Moreover, we have $k^{-1/2} \nabla u_{k} \to 0$ in $\lebe^{2}(\Omega;\R^{N \times n})$ and there holds
\begin{equation}
\label{eq:unifbound1}
\sup_{k \in \N} \big\{ \| u_k \|_{\sobo^{1,1}(\Omega;\R^{N})} + k^{-1} \| u_{k} \|_{\sobo^{1,2}(\Omega;\R^{N})}^2  \big\} < \infty.
\end{equation}
\end{corollary}

\begin{proof}
In order to prove the first claim, for a fixed number $\varepsilon > 0$, we choose first a function $v_{\varepsilon} \in \sobo^{1,2}(\Omega;\R^{N})$ and then an index $k_0 \in \N$ such that
\begin{equation*}
  \mathfrak{F}[v_{\varepsilon}] \leq  \inf_{\sobo^{1,1}(\Omega;\R^{N})} \mathfrak{F} + \frac{\varepsilon}{2} \quad \text{and} \quad (2k_0)^{-1}  \|\nabla v_{\varepsilon} \|_{\lebe^{2}(\Omega;\R^{N \times n})}^2 \leq \frac{\varepsilon}{2}
\end{equation*}
hold. In this way, we obtain by the minimality of~$u_k$ for all indices $k \geq k_0$
\begin{align*}
\inf_{\sobo^{1,1}(\Omega;\R^{N})} \mathfrak{F} & \leq \mathfrak{F}[u_{k}] \leq \mathfrak{F}[u_{k}] + (2k)^{-1}  \|\nabla u_{k} \|_{\lebe^{2}(\Omega;\R^{N \times n})}^2 = \mathfrak{F}_{k}[u_{k}] \\
  & = \inf_{\sobo^{1,2}(\Omega;\R^{N})} \mathfrak{F}_{k} \\
  & \leq \mathfrak{F}_{k} [v_{\varepsilon}] = \mathfrak{F} [v_{\varepsilon}] + (2k)^{-1}  \|\nabla v_{\varepsilon} \|_{\lebe^{2}(\Omega;\R^{N \times n})}^2 \leq \inf_{\sobo^{1,1}(\Omega;\R^{N})} \mathfrak{F} + \varepsilon\,,
\end{align*}
and the first assertion follows by arbitrariness of~$\varepsilon$. Moreover, from this chain of inequalities, we also read off the strong convergence $k^{-1/2} \nabla u_{k} \to 0$ in $\lebe^{2}(\Omega;\R^{N \times n})$. Finally, in view of $(u_{k})_{\Omega}=0$, we may apply Poincar\'{e}'s inequality in the mean value version in the spaces $\sobo^{1,2}(\Omega;\R^{N})$ and $\sobo^{1,1}(\Omega;\R^{N})$ to~$u_{k}$, and we thus infer the last claim~\eqref{eq:unifbound1} as a direct consequence of the estimate~\eqref{eq:norm_bound_u_k}.
\end{proof}

Let us note that the uniform bound~\eqref{eq:unifbound1}, Chacon's biting Lemma~\ref{lemma_biting} and the compact embedding $\sobo^{1,1}(\Omega;\R^{N})\hookrightarrow \lebe^{1}(\Omega;\R^{N})$ allows to conclude that there exist functions $u \in\bv(\Omega;\R^{N})$ with $(u)_\Omega=0$ and $E\in\lebe^{1}(\Omega;\R^{N\times n})$ such that, for a suitable non-relabelled subsequence, we have
\begin{align}
\label{eq:BVconv}
u_k \wstar u \qquad & \text{in } \bv(\Omega;\R^N), \\
u_{k}\to u \qquad & \text{in } \lebe^{1}(\Omega;\R^{N}),  \nonumber \\
\nabla u_{k} \wbit E \qquad & \text{in } \lebe^{1}(\Omega;\R^{N \times n}), \label{eq:wbit}
\end{align}
as $k \to \infty$. In order to prove the existence of a minimiser of the original functional $\mathfrak{F}$ in the space $\sobo^{1,1}(\Omega;\R^{N})$, we shall now investigate the sequence $(u_{k})_{k\in \N}$ in more detail, with the aim to get a convergence result which is more suitable for the minimisation problem~\eqref{eq:varprin}.

\subsection{A~Priori Estimates}
\label{section_a_priori_estimates}

We shall next derive suitable a~priori estimates for the sequence $(u_{k})_{k\in \N}$ which, in particular, will allow us to conclude the pointwise convergence of $(\nabla u_{k})_{k \in \N}$ to its biting-limit~$E$ almost everywhere in~$\Omega$. We begin by showing that the sequence $(u_{k})_{k \in \N}$ constructed in the previous section indeed is in $\sobo^{2,2}_{\locc}(\Omega;\R^{N})$.

\begin{lemma}\label{lem:sobolevregularity}
Consider a convex function $f\in\hold^{2}(\R_{0}^{+})$ which satisfies $f(0) = f'(0)= 0$, the linear growth condition~\eqref{eq:lingrowth} and the bound~\eqref{eq:sec_der_bound}, and let $T_{0}\in \sobo^{2,\infty}(\Omega;\R^{N\times n})$ verify~\eqref{eq:coercivcond}.  Then, for each $k \in \N$, the minimiser~$u_k$ from Lemma~\eqref{lemma_existence_uk} satisfies $u_{k}\in\sobo^{2,2}_{\locc}(\Omega;\R^{N})$, and moreover, for every compact set $K \subset \Omega$ there holds
\begin{equation}
\label{eq:maincoercive}
\sup_{k \in \N} \Big\{ \sum_{s=1}^n \int_{K} \D_z \A(\sg u_{k}) [\partial_s \sg u_{k},\partial_s \sg u_k ] \dif x + k^{-1} \int_{K} |\nabla^2 u_{k}|^2 \dif x \Big\} < \infty.
\end{equation}
\end{lemma}

\begin{proof}
Let $\eta \in \hold^{1}_0(\Omega;[0,1])$ be a localization function with $\eta \equiv 1$ on the given, compactly supported subset~$K$ of~$\Omega$. For~$h \in \R \setminus \{0\}$ with~$|h| < \dista(K,\partial\Omega)$ and $s \in \{1,\ldots,n\}$ we denote by $\Delta_{s,h}$ the finite difference quotient operator with respect to direction~$e_s$ and stepsize~$h$, and we then choose $\varphi \coloneqq  \Delta_{s,-h}(\eta^{2}\Delta_{s,h} u_{k}) \in \sobo^{1,2}(\Omega;\R^{N})$ as a test function in the Euler--Lagrange system~\eqref{eq:approxel}. In this way, we obtain with the integration by parts formula for finite difference quotients and the standard one
\begin{multline}
\label{W22_start}
\int_{\Omega} \Delta_{s,h}(\A_{k}(\nabla u_{k})) \cdot \big[ \eta^{2} \Delta_{s,h}\nabla u_{k}+2\eta \Delta_{s,h} u_{k} \otimes \nabla \eta \big] \dif x \\
   = - \int_{\Omega} \Delta_{s,h} \di T_{0} \cdot \eta^{2}\Delta_{s,h} u_{k} \dif x ,
\end{multline}
which is the starting point for the proof of higher Sobolev regularity. For the right-hand side of~\eqref{W22_start} we obtain from standard properties (regarding norm estimates) for finite difference quotients, in view of $T_{0}\in\sobo^{2,\infty}(\Omega;\R^{N\times n})$ and the uniform bound~\eqref{eq:unifbound1}, the estimate
\begin{equation}
\label{W22_rhs}
- \int_{\Omega} \Delta_{s,h} \di T_{0} \cdot \eta^{2}\Delta_{s,h} u_{k} \dif x \leq C \|T_{0}\|_{\sobo^{2,\infty}(\Omega;\R^{N\times n})} \|\nabla u_{k}\|_{\lebe^{1}(\Omega;\R^{N\times n})}\leq C
\end{equation}
with a constant~$C$ depending only on $\Omega$, $f$, $\|T_{0}\|_{\sobo^{2,\infty}(\Omega;\R^{N\times n})}$ and~$\eta$ (but independent of $k \in \N$). In order to find some coerciveness estimate for the left-hand side~\eqref{W22_start}, let us first rewrite
\begin{equation*}
 \Delta_{s,h}(\A_{k}(\nabla u_{k}(x))) = \int_{0}^{1} \D_z \A_{k} (\sg u_{k}(x)+th\Delta_{s,h}\sg u_{k}(x)) \dif t \Delta_{s,h}\sg u_{k}(x)
\end{equation*}
for $x \in K$. Thus, for shorter notation, we introduce the bilinear form ${\mathcal B}_{k,h}(x) \colon \R^{N\times n} \times \R^{N\times n} \to \R$, for all $k \in \N$, $h \in \R \setminus \{0\}$ and $x \in \Omega$ such that $\dista(x,\partial\Omega)>h$, by
\begin{align*}
{\mathcal B}_{k,h}(x)[\zeta,\tilde{\zeta}]\coloneqq \int_{0}^{1} \D_z \A_{k} (\sg u_{k}(x)+th\Delta_{s,h}\sg u_{k}(x)) [\zeta , \tilde{\zeta}] \dif t \qquad \text{for } \zeta,\tilde{\zeta} \in\R^{N\times n}.
\end{align*}
Note that, by definition, the radial structure and due to the convexity of~$f$ with $f'(0)=0$, these bilinear forms are (for all $k$, $h$ and $x$ as above) symmetric and positive definite, with lower bound ${\mathcal B}_{k,h}(x)[\zeta,\zeta] \geq k^{-1} |\zeta|^2$ for all $\zeta \in\R^{N\times n}$. Consequently, applying Young's inequality in the bilinear forms ${\mathcal B}_{k,h}(x)$ and invoking~\eqref{W22_rhs}, we deduce from~\eqref{W22_start} the estimate
\begin{align*}
\lefteqn{\int_{\Omega} \eta^{2}{\mathcal B}_{k,h}(x)[\Delta_{s,h}\sg u_{k},\Delta_{s,h}\sg u_{k}]\dif x} \\
 & = - 2\int_{\Omega}{\mathcal B}_{k,h}(x)[\Delta_{s,h}\sg u_{k},\eta \Delta_{s,h}u_{k} \otimes  \nabla \eta ] \dif x - \int_{\Omega} \Delta_{s,h} \di T_{0} \cdot \eta^{2}\Delta_{s,h} u_{k} \dif x \\
 & \leq \frac{1}{2} \int_{\Omega} \eta^{2} {\mathcal B}_{k,h}(x)[\Delta_{s,h}\sg u_{k},\Delta_{s,h}\sg u_{k}]\dif x \\
 &  \quad + 2 \int_{\Omega} {\mathcal B}_{k,h}(x)[ \Delta_{s,h}u_{k} \otimes  \nabla\eta ,\Delta_{s,h}u_{k} \otimes \nabla\eta ]\dif x + C.
\end{align*}
We may now absorb the first term of the right-hand side into the left-hand side. By~\eqref{eq:growth_D_z_A} in conjunction with~\eqref{eq:Aregularised}, by standard properties of finite difference quotients and by~\eqref{eq:unifbound1} we then obtain
\begin{align*}
k^{-1}\int_{\Omega}\eta^{2} |\Delta_{s,h}\sg u_{k}|^2 \dif x & \leq \int_{\Omega}\eta^{2}{\mathcal B}_{k,h}(x)[\Delta_{s,h}\sg u_{k},\Delta_{s,h}\sg u_{k}]\dif x  \\
  &  \leq  C \int_{\Omega} | \nabla \eta|^2 |\Delta_{s,h}u_{k}(x)|^{2}\dif x + C \leq C
\end{align*}
for a constant~$C$ depending only on $\Omega$, $f$, $\|T_{0}\|_{\sobo^{2,\infty}(\Omega;\R^{N\times n})}$, $\eta$ and~$k$. By choice of the localization function $\eta$ we thus obtain, for each $k \in \N$, that $\Delta_{s,h}\sg u_{k}$ is bounded uniformly for all $h \in \R \setminus \{0\}$ with~$|h| < \dista(K,\partial\Omega)$ in $\lebe^2(K; \R^{N \times n})$, though not uniformly in~$k$. The $\sobo^{2,2}_{\locc}$-regularity of $u_k$ then follows from the usual difference-quotient type characterisation of $\sobo^{1,2}$ and the arbitrariness of~ the compact set $K \subset \Omega$ and of $s \in \{1,\ldots, n\}$.

Once the $\sobo^{2,2}_{\locc}$-regularity of each function~$u_k$ is at our disposal, we may now proceed to the proof of the uniform estimate. To this end, we first differentiate the Euler--Lagrange system~\eqref{eq:approxel} and repeat essentially the same computations as above, but now with the differential~$\partial_s$ instead of the difference quotient operator~$\Delta_{s,h}$. More precisely, starting from the
identity
\begin{equation*}
 \int_{\Omega} \D_z \A_{k}(\sg u_{k}) [ \partial_s \sg u_{k} , \sg \varphi] \dif x = \int_{\Omega} \partial_s T_{0} \cdot \sg \varphi \dif x
\end{equation*}
for all functions $\varphi \in W^{1,2}(\Omega;\R^{N})$ with compact support in~$\Omega$, we choose $\varphi = \eta^2 \partial_s u$ with $\eta  \in \hold^{1}_0(\Omega;[0,1])$ a localization function on the compact set $K \subset \Omega$ as above. Doing so, we find via Young's inequality (applied to the positive definite bilinear forms $\D_z \A_{k}(\sg u_{k}(x))$ corresponding to ${\mathcal B}_{k,0}(x)$ above) and the integration by parts formula
\begin{align*}
\lefteqn{\int_{\Omega} \eta^2 \D_z \A_{k}(\sg u_{k}) [\partial_s \sg u_{k}, \partial_s \sg u_k ] \dif x  } \\
  & = -2 \int_{\Omega} \eta \D_z \A_{k} \sg u_{k}) [ \partial_s \sg u_{k}, \partial_s u_k \otimes \sg \eta ]  \dif x
  + \int_{\Omega} \partial_s T_{0} \cdot \nabla( \eta^2 \partial_s u)  \dif x \\
  & \leq \frac{1}{2} \int_{\Omega} \eta^2 \D_z \A_{k}(\sg u_{k}) [ \partial_s \sg u_{k}, \partial_s \sg u_k] \dif x \\
  & \quad + 2 \int_{\Omega} \eta^2 \D_z \A_{k}(\sg u_{k}) [\partial_s u_k \otimes \sg \eta, \partial_s u_k \otimes \sg \eta] \dif x - \int_{\Omega} \eta^2 \partial_s \di T_{0} \cdot \partial_s u_k \dif x
\end{align*}
After absorbing the first integral on the right-hand side into the left-hand side, we directly obtain the lower bound given in the statement via the definition~\eqref{eq:Aregularised} of~$\A_k$, while the remaining terms on the right-hand side of the previous inequality are estimated via~\eqref{eq:growth_D_z_A}, combined with~\eqref{eq:Aregularised} and $T_0 \in \sobo^{2,\infty}(\Omega; \R^{N \times n})$. This yields
\begin{align*}
\lefteqn{  \int_{\Omega} \eta^2 \D_z \A(\sg u_{k}) [ \partial_s \sg u_{k}, \partial_s \sg u_k]\dif x + k^{-1} \int_{\Omega} \eta^2|\partial_s \sg u_{k}|^2 \dif x } \\
  & \leq \int_{\Omega} \eta^2 \D_z \A_{k}(\sg u_{k}) [ \partial_s \sg u_{k} , \partial_s \sg u_k ] \dif x \\
  & \leq C \big( \| \partial_s u_k \|_{\lebe^{1}(\Omega;\R^{N})} + k^{-1} \| \partial_s u_k \|_{\lebe^{2}(\Omega;\R^{N})}^2 \big)
\end{align*}
with a constant~$C$ depending only on~$L$, $\|T_{0}\|_{\sobo^{2,\infty}(\Omega;\R^{N\times n})}$ and $\eta$, but not on~$k$. At this stage, the assertion~\eqref{eq:maincoercive} of the lemma follows from the uniform bound~\eqref{eq:unifbound1}, combined with the arbitrariness of $s \in \{1,\ldots,n\}$.
\end{proof}

\begin{remark}
Invoking the condition~\eqref{eq:hmonotone} of $h$-monotonicity satisfied by the integrand with  $h>0$ almost everywhere on~$\R_{0}^{+}$, we can interpret the uniform estimate~\eqref{eq:maincoercive} as a weighted Sobolev-type estimate, namely that we have, for every compact set $K \subset \Omega$,
\begin{equation}
\label{eq:maincoercive_h}
\sup_{k \in \N} \Big\{ \int_{K} h(|\sg u_{k}|) |\nabla^2 u_{k}|^2 \dif x \Big\} < \infty.
\end{equation}
\end{remark}

The uniform bound~\eqref{eq:maincoercive_h} constitutes the key ingredient in order to establish the pointwise convergence of the gradients $(\nabla u_k)_{k \in \N}$.

\begin{corollary}
\label{cor_pointwise_convergence}
If the assumptions of the previous Lemma~\ref{lem:sobolevregularity} is satisfied and $f$ is strictly convex, then we have
\begin{equation}
\label{eq:pointwise0}
\nabla u_{k}\to E\qquad\mathscr{L}^{n}\text{-a.e.~in }\Omega\quad\text{as } k \to \infty,
\end{equation}
where $E\in\lebe^{1}(\Omega;\R^{N\times n})$ is given by the biting limit~\eqref{eq:wbit}.
\end{corollary}

\begin{proof}
We here follow the strategy of proof of \cite[Section~4.4]{BECBULMALSUE17}. We start by defining an auxiliary function $\tilde{h} \in \hold^1(\R^+,\R^+)$ via
\begin{equation*}
\tilde{h}(t)\coloneqq \int_{t}^{\infty}\frac{h(\tau)}{1+\tau}\dif \tau,\qquad \text{for } t>0,
\end{equation*}
where the function~$h$ was introduced in Remark~\ref{rem_h_monotone}. Since~$h$ is almost everywhere positive, $\tilde{h}$ is strictly monotonically decreasing and, moreover, since~$h$ satisfies~\eqref{eq:growth_D_z_A}, we have
\begin{equation*}
\tilde{h}(t) \leq \int_{t}^{\infty}\frac{2L}{(1+\tau)^{2}}\dif \tau = 2L (1+t)^{-1} \qquad\text{for }t>0.
\end{equation*}
Next, we introduce the functions
\begin{equation*}
\alpha_{k}\coloneqq A(\sg u_k) \qquad \text{and} \qquad \beta_{k}\coloneqq \tilde{h}(|\sg u_{k}|)
\end{equation*}
for $k \in \N$. Obviously, $\alpha_k$ and $\beta_k$ are bounded in~$\Omega$. Next, we observe from Cauchy--Schwarz inequality for each $s \in \{1,\ldots,n\}$ 
\begin{align}
 | \partial_s \alpha_k|^2 & =  \D_z \A(\sg u_{k}) [ \partial_s \sg u_{k}, \partial_s \alpha_k] \label{CS_sec_der} \\
  & \leq  \big( \D_z \A(\sg u_{k}) [ \partial_s \sg u_{k}, \partial_s \sg u_{k} ] \big)^{\frac{1}{2}} \big( \D_z \A(\sg u_{k}) [ \partial_s  \alpha_k, \partial_s  \alpha_k ] \big)^{\frac{1}{2}} \nonumber \\
  & \leq L^{\frac{1}{2}}  \big( \D_z \A(\sg u_{k}) [ \partial_s \sg u_{k}, \partial_s \sg u_{k} ] \big)^{\frac{1}{2}} | \partial_s \alpha_k| \nonumber
\end{align}
and thus
\begin{equation*}
 | \nabla \alpha_k|^2 \leq L \sum_{s=1}^n \D_z \A(\sg u_{k}) [ \partial_s \sg u_{k}, \partial_s \sg u_{k} ] ,
\end{equation*}
while from the definition of~$\tilde{h}$ and the bound on~$h$ we directly get 
\begin{equation*}
 | \nabla \beta_k |^2 \leq  2L h(|\sg u_{k}|)  |\nabla^2 u_{k}|^2.
\end{equation*}
In conclusion, by~\eqref{eq:growth_D_z_A} we have shown
\begin{equation*}
 | \nabla \alpha_k |^2 +| \nabla \beta_k |^2 \leq  3L \sum_{s=1}^n  \D_z \A(\sg u_{k}) [ \partial_s \sg u_{k}, \partial_s \sg u_k],
\end{equation*}
and Lemma~\ref{lem:sobolevregularity} thus yields
\begin{equation*}
\sup_{k \in \N} \Big\{ \|\alpha_{k}\|_{\lebe^{\infty}(\Omega;\R^{N \times n})}+\|\beta_{k}\|_{\lebe^{\infty}(\Omega)} + \| \alpha_{k}\|_{\sobo^{1,2}(K;\R^{N \times n})} + \| \beta_{k}\|_{\sobo^{1,2}(K)} \Big\} < \infty
\end{equation*}
for each compact set $K \subset \Omega$. If~$K$ has a Lipschitz boundary, we find, thanks to the compact embedding $\sobo^{1,2}(K;\R^{N \times n}) \hookrightarrow\lebe^{1}(K;\R^{N \times n})$, non-relabelled subsequences such that the following convergence results hold:
\begin{align*}
& \alpha_{k}\rightharpoonup \alpha \qquad \text{weakly in } \lebe^{1}(\Omega;\R^{N \times n}),\\
& \alpha_{k}\to \alpha\qquad\text{strongly in }\lebe^{1}(K;\R^{N \times n}),\\
& \alpha_{k}\to \alpha\qquad\mathscr{L}^{n}\text{-a.e.~in } \Omega\\
& \beta_{k}\rightharpoonup \beta \qquad \text{weakly in } \lebe^{1}(\Omega),\\
& \beta_{k}\to \beta\qquad\text{strongly in }\lebe^{1}(K),\\
& \beta_{k}\to\beta\qquad\mathscr{L}^{n}\text{-a.e.~in }\Omega.
\end{align*}
Since $\tilde{h}$ is strictly decreasing on $\R^+$, the inverse $\tilde{h}^{-1}$ exists on the set $\tilde{h}(\R^+)$, is non-negative, decreasing and continuous. Thus, in view of Fatou's lemma and the boundedness of $(\sg u_{k})_{k \in \N}$ in $\lebe^{1}(\Omega;\R^{N \times n})$ by~\eqref{eq:unifbound1}, we get
\begin{equation*}
\int_{\Omega}\tilde{h}^{-1}(\beta)\dif x \leq \liminf_{k \to \infty}\int_{\Omega}\tilde{h}^{-1}(\beta_{k})\dif x = \liminf_{k \to \infty}\int_{\Omega} |\sg u_{k}| \dif x  < \infty .
\end{equation*}
With $\lim_{t \to \infty} \tilde{h}(t) = 0$ and thus $\lim_{t \to 0} \tilde{h}^{-1}(t)=\infty$, we easily deduce that $\beta>0$ and $0 < \tilde{h}^{-1}(\beta)<\infty$ holds $\mathscr{L}^{n}$-a.e.~in $\Omega$. Therefore, due to the continuity of $t/f'(t)$, we have on the one hand the pointwise convergence
\begin{equation*}
\sg u_{k} = \frac{A(\sg u_k) |\sg u_k|}{f'(|\sg u_k|)} = \frac{\alpha_{k} \tilde{h}^{-1}(\beta_k)}{f'(\tilde{h}^{-1}(\beta_{k}))}\to \frac{\alpha \tilde{h}^{-1}(\beta)}{f'(\tilde{h}^{-1}(\beta))} \qquad\mathscr{L}^{n}\text{-a.e.~in } \Omega \text{ as } k \to \infty.
\end{equation*}
On the other hand,~\eqref{eq:wbit} yields the existence of an increasing sequence $(\Omega_{j})_{j \in \N}$ of sets contained in~$\Omega$ with $\mathscr{L}^{n}(\Omega\setminus\Omega_{j})\to 0$ as $j\to\infty$ and such that $\sg u_{k}$ converges weakly to~$E$ as $k\to\infty$ on every~$\Omega_{j}$. Therefore, because of uniqueness of the limits, we can identify the pointwise limit $\alpha  \tilde{h}^{-1}(\beta)/ f'(\tilde{h}^{-1}(\beta)) = E$ as $\lebe^1(\Omega;\R^{N \times n})$ functions. In conclusion, we arrive at the convergence $\sg u_{k} \to E$ $\mathscr{L}^{n}$-a.e.~in~$\Omega$, which was the claim~\eqref{eq:pointwise0}. Moreover, once again by Fatou's lemma, combined with the uniform bound~\eqref{eq:unifbound1}, we also have the estimate
\begin{equation*}
\|E\|_{\lebe^{1}(\Omega;\R^{N \times n})}\leq \liminf_{k \to \infty}\|\sg u_{k}\|_{\lebe^{1}(\Omega;\R^{N \times n})}\leq C. \qedhere
\end{equation*}
\end{proof}

\subsection{Existence and Regularity for the Primal Problem}

We shall now use the a~priori estimates of the preceding sections to conclude that there exists a function $v\in\sobo^{1,1}(\Omega;\R^{N})$ such that~$E$ --~given by the biting limit~\eqref{eq:wbit} and which was just identified in Corollary~\ref{cor_pointwise_convergence} as the pointwise limit of the sequence $(\nabla u_k)_{k \in \N}$~-- satisfies
\begin{align}\label{eq:pointwise1}
\sg u_{k}\to E = \sg v\qquad\mathscr{L}^{n}\text{-a.e.~in }\Omega\quad\text{as } k \to \infty.
\end{align}

\begin{proof}[Proof of the representation $E=\sg v$]
We shall utilize the Poincar\'{e}-type Lemma~\ref{lem:Poinclemma} (applied to the $N$~component functions of $E$, each of them with values in~$\R^n$). Hence, in what follows, we want to prove that every function~$E^\alpha$, for $\alpha \in \{1,\ldots,N\}$, is curl-free in the sense of distributions, as introduced in Definition~\ref{def_curl_free}. This means that we need to show
\begin{equation}
\label{eq_representation_curl_free}
 \int_\Omega \big( E_i \partial_j \varphi - E_j \partial_i \varphi) \dif x = 0
\end{equation}
for any fixed test function $\varphi \in \hold^{1}_0(\Omega)$ and all choices of indices $i,j \in \{1,\ldots,n\}$. To this end, we set $K\coloneqq \spt(\varphi)$. We further consider a sequence of functions $(g_{\ell})_{\ell \in \N}$ in $\hold_{c}^{\infty}(\R;[0,1])$ with $g_{\ell} \equiv 1$ in $[-\ell,\ell]$, $g_{\ell}\equiv 0$ outside of $[-2\ell,2\ell]$ and $|g'_{\ell}|\leq 2 \ell^{-1}$ in $\R$, which allows us to estimate the above expression on sublevel sets of~$|E^\alpha|$. In fact, we may now rewrite the expression in~\eqref{eq_representation_curl_free} above as
\begin{multline*}
\left\vert \int_{\Omega} \big( E^\alpha_i \partial_j \varphi - E^\alpha_j \partial_i \varphi) \dif x\right\vert
  = \left\vert \int_{\Omega} g_{\ell}(|E|) \big( E^\alpha_i \partial_j \varphi - E^\alpha_j \partial_i \varphi) \dif x\right\vert \\
  + \left\vert \int_{\Omega} \big( 1 - g_{\ell}(|E|) \big) \big( E^\alpha_i \partial_j \varphi - E^\alpha_j \partial_i \varphi) \dif x \right\vert
  =: \mathbf{I}_{\ell}+\mathbf{II}_{\ell},
\end{multline*}
and noting that $E^\alpha \in\lebe^{1}(\Omega;\R^{n})$, we find
\begin{align*}
\lim_{\ell \to \infty} \mathbf{II}_{\ell} \leq 2 \sup_{K}|\nabla \varphi| \lim_{\ell \to\infty}\int_{\{|E|\geq \ell\}}|E^\alpha|\dif x = 0.
\end{align*}
Thus, it remains to show that we also have $\lim_{\ell \to \infty} \mathbf{I}_{\ell} = 0$. In order to prove this claim, we start by observing that, as a consequence of Lebesgue's dominated convergence theorem, the pointwise convergence $\sg u_{k} \to E$ established in Corollary~\ref{cor_pointwise_convergence} implies the strong convergence $g_{\ell}(|\nabla u_{k}|)\nabla u_{k}^\alpha \to g_{\ell}(|E|) E^\alpha$ in $\lebe^{1}(\Omega;\R^{n})$ as $k \to \infty$. Since by Lemma~\ref{lem:sobolevregularity} we have $u_{k} \in \sobo_{\locc}^{2,2}(\Omega;\R^{N})$ for every $k \in \N$, we may hence rewrite $\mathbf{I}_{\ell}$ by the integration by parts formula as
\begin{align*}
\mathbf{I}_{\ell} & = \lim_{k \to \infty} \bigg\vert\int_{\Omega} g_{\ell}(|\nabla u_{k}|) \big( \partial_i u^\alpha_k \partial_j \varphi - \partial_j u^\alpha_k \partial_i \varphi)  \dif x \bigg\vert \\
 & = \lim_{k \to \infty} \bigg\vert \int_{\Omega} \Big( \partial_j (g_{\ell}(|\nabla u_{k}|)) \partial_i u_{k}^\alpha - \partial_i (g_{\ell}(|\nabla u_{k}|)) \partial_j u_{k}^\alpha\Big) \varphi \dif x \\
 & \hspace{4cm} + \int_{\Omega} g_{\ell}(|\nabla u_{k}|) \big( \partial_j \partial_i u_{k}^\alpha - \partial_i \partial_j u_{k}^\alpha \big) \varphi \dif x \bigg\vert \\
 & = \lim_{k \to \infty} \bigg\vert \int_{\Omega} \Big( \partial_j (g_{\ell}(|\nabla u_{k}|)) \partial_i u_{k}^\alpha - \partial_i (g_{\ell}(|\nabla u_{k}|)) \partial_j u_{k}^\alpha\Big) \varphi \dif x \bigg\vert .
\end{align*}
We next introduce functions $G_{\ell}\colon\R_{0}^+ \to\R$ by
\begin{align*}
G_{\ell}(t)\coloneqq \int_{0}^{t}\frac{g'_{\ell}(\tau) \tau}{f'(\tau)}\dif \tau,\qquad \textrm{for } t \geq 0 \text{ and } \ell \in\mathbb{N}.
\end{align*}
Firstly, since~$f$ is strictly convex with $f'(0)=0$, we note that~$f'$ is monotonously increasing with $f'(t)>0$ for all $t > 0$. Consequently, the integrand in the definition of~$G_{\ell}$ is well-defined and supported in $[\ell,2\ell]\subset \R^+$, and we further have the estimate
\begin{align}\label{eq:Gbound}
|G_{\ell}(t)|\leq \frac{4 \ell}{\ell f'(\ell)} \int_{\ell}^{2\ell} 1 \dif \tau \leq \frac{4\ell}{f'(\ell)} \leq \frac{4 \ell}{f'(1)}  \qquad\text{for all } t>0 \text{ and } \ell \in\mathbb{N}.
\end{align}
Using
\begin{equation*}
\partial_j (g_\ell(|\nabla u_k|)) = \partial_j (G_\ell(|\nabla u_k|)) \frac{f'(|\nabla u_k|)}{|\nabla u_k|},
\end{equation*}
we may then express~$\mathbf{I}_{\ell}$ in terms of $G_{\ell}(|\nabla u_k|)) $ and apply once again the integration by parts formula (as well as the fact that $u_{k} \in \sobo_{\locc}^{2,2}(\Omega;\R^{N})$ holds for each $k \in \N$). In this way, we find
\begin{align*}
\mathbf{I}_{\ell} &  = \lim_{k \to \infty}\left\vert \int_{\Omega} \Big( \partial_{j}\big(G_{\ell}(|\nabla u_{k}|)\big) \partial_i u_{k}^\alpha - \partial_{i}\big(G_{\ell}(|\nabla u_{k}|)\big) \partial_j u_{k}^\alpha \Big) \frac{f'(|\nabla u_k|)}{|\nabla u_k|} \varphi \dif x \right\vert \\
 & \leq \lim_{k \to \infty}\left\vert \int_{\Omega} G_{\ell}(|\nabla u_{k}|) \bigg( \partial_{j} \Big( \frac{f'(|\nabla u_k|) \partial_i u_{k}^\alpha}{|\nabla u_k|} \Big)-  \partial_{i} \Big( \frac{f'(|\nabla u_k|) \partial_j u_{k}^\alpha}{|\nabla u_k|} \Big) \bigg) \varphi \dif x \right\vert\\
 & \quad + \lim_{k\to\infty}\left\vert\int_{\Omega}G_{\ell}(|\nabla u_{k}|)\big( \partial_i u_{k}^\alpha \partial_{j}\varphi- \partial_j u_{k}^\alpha \partial_{i}\varphi \big) \frac{f'(|\nabla u_k|)}{|\nabla u_k|} \dif x \right\vert .
\end{align*}
Recalling
\begin{equation*}
A(z) = \frac{f'(|z|) z}{|z|} \qquad \text{for all } z\in\R^{N\times n},
\end{equation*}
we next estimate~$\mathbf{I}_{\ell}$ in the more convenient form
\begin{align*}
\mathbf{I}_{\ell}
 & \leq \lim_{k\to\infty} \left\vert \int_{\Omega} \D_z \A(\sg u_{k}) \big[ \partial_j \sg u_{k}, G_{\ell}(|\nabla u_{k}|)  e^\alpha  \otimes e_i \big] \varphi \dif x \right\vert\\
 & \quad +  \lim_{k\to\infty} \left\vert \int_{\Omega} \D_z \A(\sg u_{k}) \big[ \partial_i \sg u_{k}, G_{\ell}(|\nabla u_{k}|) e^\alpha \otimes e_j \big] \varphi \dif x \right\vert\\
 & \quad + \lim_{k\to\infty} 2 \int_{\Omega} |G_{\ell}(|\nabla u_{k}|)| |f'(|\nabla u_{k}|)| |\nabla \varphi| \dif x,
\end{align*}
where $e_1,\ldots,e_n$ denote the standard unit basis vectors in~$\R^n$ and $e^1,\ldots,e^N$ the ones in~$\R^N$. Keeping in mind that $\D_z \A(z)$ is a positive definite, symmetric bilinear form, we infer from the Cauchy--Schwarz inequality
\begin{align*}
\lefteqn{ \left\vert \int_{\Omega} \D_z \A(\sg u_{k}) \big[ \partial_j \sg u_{k}, G_{\ell}(|\nabla u_{k}|) e^\alpha  \otimes e_i \big] \varphi \dif x \right\vert}\\
 & \leq  \bigg( \int_{\Omega} \D_z \A(\sg u_{k}) \big[ \partial_j \sg u_{k},\partial_j \sg u_{k} \big] |\varphi| \dif x \bigg)^{\frac{1}{2}} \\
 & \hspace{1cm} \times  \bigg( \int_{\Omega} \D_z \A(\sg u_{k}) \big[ G_{\ell}(|\nabla u_{k}|) e^\alpha  \otimes e_i , G_{\ell}(|\nabla u_{k}|) e^\alpha  \otimes e_i \big] |\varphi| \dif x \bigg)^{\frac{1}{2}}
\end{align*}
(and analogously with~$i$ replaced by~$j$). Thus, employing the a~priori estimate~\eqref{eq:maincoercive} from Lemma~\ref{lem:sobolevregularity} (note $K = \spt(\varphi) \Subset \Omega$), the upper bound in~\eqref{eq:growth_D_z_A}, the boundedness of~$f'$ by~$L$ and the growth~\eqref{eq:Gbound} as well as the support of~$G_\ell$, we arrive at
\begin{align*}
\mathbf{I}_{\ell} & \leq C \lim_{k \to \infty}  \bigg( \int_{\Omega} (1 + |\sg u_{k}|)^{-1} |G_{\ell}(|\nabla u_{k}|)|^2 |\varphi| \dif x \bigg)^{\frac{1}{2}} \\
 & \quad + C \lim_{k\to\infty} 2 \int_{\Omega} |G_{\ell}(|\nabla u_{k}|)| |\nabla \varphi| \dif x\\
 & \leq C \lim_{k\to\infty}\Phi\bigg( \int_{\{|\nabla u_{k}|\geq \ell\}} \ell \dif x \bigg),
\end{align*}
with $\Phi\colon\R_{0}^{+}\to\R_{0}^{+}$ given by $\Phi(t)\coloneqq \max\{t^{\frac{1}{2}},t\}$ and a constant~$C$ depending only on the data and~$\varphi$, but not on~$\ell$. Finally, the pointwise convergence $\sg u_{k} \to E$ allows us to pass to the limit $k \to \infty$, which yields
\begin{equation*}
 \mathbf{I}_{\ell} \leq C \Phi\bigg(\int_{\{|E|\geq \ell\}} |E| \dif x \bigg).
\end{equation*}
In view of the integrability of~$E$, this proves $\lim_{\ell \to \infty} \mathbf{I}_{\ell} = 0$. In conclusion, since $\alpha \in \{1,\ldots,N\}$ was arbitrary, we have shown the claim~\eqref{eq_representation_curl_free}, i.e., that $E\in\lebe^{1}(\Omega;\R^{N\times n})$ is curl-free in the sense of distributions. Thus, as~$\Omega$ is a simply connected Lipschitz domain, Lemma~\ref{lem:Poinclemma} provides a mapping $v\in\sobo^{1,1}(\Omega;\R^{N})$ with $\nabla v = E$, and the proof of the representation is complete.
\end{proof}

\begin{remark}
\label{rem:curl_free_general}
In case that~$\Omega$ is not simply connected, we still obtain that the pointwise limit of the sequence $(\nabla u_k)_{k \in \N}$ is curl-free in the sense of distributions, but we cannot identify it as the gradient of a $\sobo^{1,1}(\Omega;\R^N)$-function.
\end{remark}

For the sake of completeness, we now proceed by demonstrating that $v\in\sobo^{1,1}(\Omega;\R^{N})$ --~after translation by $(v)_\Omega$~-- is actually a solution to the system~\eqref{eq:PDE} subject to the Neumann condition~\eqref{eq:neumann}. To this end, we firstly provide the

\begin{proof}[Proof of the uniqueness assertion of Theorem~\ref{thm:main}]
We suppose that there exist two solutions $u_1, u_2 \in\sobo^{1,1}(\Omega;\R^{N})$ to the system~\eqref{eq:PDE} subject to~\eqref{eq:neumann}, with $(u_1)_\Omega= (u_2)_\Omega = 0$ and $u_1 \neq u_2$ as $\lebe^1(\Omega;\R^{N})$~functions, which, by connectedness of~$\Omega$, also implies $\nabla u_1 \neq \nabla u_2$ as $\lebe^1(\Omega;\R^{N \times n})$~functions. In view of Lemma~\ref{lemma:weaksolution_variational},~$u_1$ and~$u_2$ both solve the variational problem~\eqref{eq:varprin}, i.e., they both minimise~$\mathfrak{F}$ in $\sobo^{1,1}(\Omega;\R^{N})$. Choosing $(u_1+u_2)/2 \in \sobo^{1,1}(\Omega;\R^{N})$ as competitor, we deduce from the strict convexity of~$f$ combined with the minimality of~$u_1$ and~$u_2$
\begin{align*}
\mathfrak{F}\Big[\frac{u_1+u_2}{2}\Big] < \frac{1}{2} \big(\mathfrak{F}[u_1]+\mathfrak{F}[u_2]\big)=\inf_{\sobo^{1,1}(\Omega;\R^{N})}\mathfrak{F},
\end{align*}
which is a contradiction. Thus the proof of uniqueness is complete.
\end{proof}

We shall now conclude the proof of Theorem~\ref{thm:main} by the

\begin{proof}[Proof of the solution property of~$v - (v)_\Omega$]
By Corollary~\ref{cor_uniform_bound_uk}, we first note that $(u_{k})_{k \in \N}$ is a minimising sequence for $\mathfrak{F}$. Next, by the pointwise convergence~\eqref{eq:pointwise1} we obtain
\begin{equation*}
f(|\nabla u_{k}|)- T_{0} \cdot \nabla u_{k} \to f(|\nabla v|)- T_{0} \cdot \nabla v \qquad\mathscr{L}^{n}\text{-a.e.~in }\Omega\quad\text{as } k \to \infty.
\end{equation*}
By the coerciveness condition~\eqref{eq:coercivcond}, which in turn implies the boundedness of the map $z \mapsto f(|z|)- T_{0} \cdot z$ from below, we thus deduce by the generalised version of Fatou's Lemma
\begin{equation*}
\mathfrak{F}[v- (v)_\Omega] = \mathfrak{F}[v]\leq \liminf_{k\to\infty}\mathfrak{F}[u_{k}] = \inf_{\sobo^{1,1}(\Omega;\R^{N})}\mathfrak{F}.
\end{equation*}
In conclusion, we have shown that~$v - (v)_\Omega$ is a minimiser with vanishing mean value in~$\Omega$, and taking advantage of Lemma~\ref{lemma:weaksolution_variational}, it is also the desired weak solution to the system~\eqref{eq:PDE} subject to~\eqref{eq:neumann}. This completes the proof of Theorem~\ref{thm:main}.
\end{proof}

Finally, we note that the solution~$v - (v)_\Omega$ is precisely the function~$u$ from~\eqref{eq:BVconv}, namely the strong $L^1(\Omega;\mathbb{R}^{N})$- and weak-$*$ $\bv(\Omega;\mathbb{R}^N)$-limit of the minimising sequence $(u_{k})_{k \in \N}$.

\begin{corollary}
\label{cor_strong_convergence}
If the assumptions of Theorem~\ref{thm:main} are satisfied, then the minimising sequence $(u_k)_{k \in \N}$ constructed in Lemma~\ref{lemma_existence_uk}  converges to $v - (v)_\Omega$ strongly in $\sobo^{1,1}(\Omega)$.
\end{corollary}

\begin{proof}
Since $u_k$ has zero mean value over~$\Omega$ for each $k \in \N$, it is enough to prove that 
\begin{equation}\label{claim}
\sg  u_k \to \sg v \qquad \textrm{strongly in }  L^1(\Omega; \mathbb{R}^{N\times n}) \quad\text{as } k \to \infty.
\end{equation}
First, thanks to the assumption~\eqref{eq:coercivcond}, we can define functions
\begin{equation*}
g_k \coloneqq \sqrt {f^{\infty}(1)|\sg u_k| - T_0 \cdot \sg u_k}.
\end{equation*}
Then, using~\eqref{eq:wbit} and~\eqref{eq:pointwise1}, we observe that 
\begin{equation}\label{weak:z}
g_k \rightharpoonup g \coloneqq \sqrt {f^{\infty}(1)|\sg v| - T_0 \cdot \sg v} \qquad \textrm{weakly in } L^2(\Omega).
\end{equation}
Our first goal is to show that 
\begin{equation}\label{strong:z}
g_k \to g \qquad \textrm{strongly in } L^2(\Omega).
\end{equation}
For this purpose, we start by recalling two identities, namely by setting $\varphi \coloneqq u_k$ in the Euler--Lagrange system~\eqref{eq:approxel} for the approximate problem and by further using the fact that~$v$ is a weak solution to the Euler--Lagrange system~\eqref{eq:EL_system} with $\varphi \coloneqq v$ we obtain
\begin{equation*}
\int_{\Omega} \big[ A_k (\sg u_k) \cdot \sg u_k - T_0 \cdot \sg u_k \big] \dif x = 0 = \int_{\Omega} \big[  A (\sg v) \cdot \sg v - T_0 \cdot \sg v \big] \dif x
\end{equation*}
for each $k \in \N$. With these identities and the definitions of $A_k$ and $A$, respectively, it is straight forward to deduce 
\begin{align*}
&\limsup_{k\to \infty} \|g_k\|_{\lebe^2(\Omega)}^2 \\
&\le \limsup_{k\to \infty} \int_{\Omega} \big[ f^{\infty}(1)|\sg u_k| - T_0 \cdot \sg u_k + A_k(\sg u_k) \cdot \sg u_k- f'(|\sg u_k|)|\sg u_k| \big] \dif x\\
&= \limsup_{k\to \infty} \int_{\Omega} \big[ f^{\infty}(1)|\sg u_k| - T_0 \cdot \sg v + A(\sg v) \cdot \sg v - f'(|\sg u_k|)|\sg u_k| \big] \dif x\\
&=  \|g \|_{\lebe^2(\Omega)}^2 + \limsup_{k\to \infty} \int_{\Omega} \big[ f^{\infty}(1)|\sg u_k| - f'(|\sg u_k|)|\sg u_k|-f^{\infty}(1)|\sg v| + f'(|\sg v|)|\sg v| \big] \dif x.
\end{align*}
In addition, thanks to~\eqref{eq:pointwise1}, we also have 
\begin{equation*}
f^{\infty}(1)|\sg u_k| - f'(|\sg u_k|)|\sg u_k| \to f^{\infty}(1)|\sg v| - f'(|\sg v|)|\sg v| \quad\mathscr{L}^{n}\text{-a.e.~in $\Omega$  as } k \to \infty.
\end{equation*}
Thus, if the above sequence is uniformly integrable, then by the Vitali convergence theorem we get
\begin{equation*}
\limsup_{k\to \infty}  \|g_k\|_{\lebe^2(\Omega)}^2 \le  \|g\|_{\lebe^2(\Omega)}^2,
\end{equation*}
which together with~\eqref{weak:z} implies~\eqref{strong:z}. For proving uniform integrability, we fix $\varepsilon>0$ and determine $\lambda>0$ such that 
\begin{equation*}
f^{\infty}(1) - f'(\lambda) = \lim_{t\to \infty} f'(t) - f'(\lambda) \le \varepsilon. 
\end{equation*}
Then for every set $U\subset \Omega$ fulfilling $\mathscr{L}^{n}(U)\le \varepsilon/(f^{\infty}(1)\lambda)$, we obtain by monotonicity of~$f'$
\begin{align*}
\lefteqn{ \int_{U} \big[ f^{\infty}(1)|\sg u_k| - f'(|\sg u_k|)|\sg u_k| \big] \dif x } \\
&= \int_{U\cap \{|\sg u_k|\le \lambda\}} \big[ f^{\infty}(1)|\sg u_k| - f'(|\sg u_k|)|\sg u_k| \big] \dif x\\
&\qquad +\int_{U\cap \{|\sg u_k|> \lambda\}} \big[ f^{\infty}(1)|\sg u_k| - f'(|\sg u_k|)|\sg u_k| \big] \dif x \\
&\le f^{\infty}(1)\lambda \mathscr{L}^{n}(U) + \big( f^{\infty}(1) - f'(\lambda) \big) \int_{U}|\sg u_k|\dif x \le C\varepsilon,
\end{align*}
where we also used the a~priori bound~\eqref{eq:unifbound1}. Hence, we have uniform integrability and the proof of the strong convergence~\eqref{strong:z} is complete.  

Now, with $g_k \to g$ converging strongly in $\lebe^2(\Omega)$ as $k \to \infty$, the sequence $(g_k^2)_{k \in \N}$ is uniformly integrable and then, thanks to $g_k^2 \geq (f^\infty(1) - \| T_0 \|_{\lebe^\infty(\Omega;\mathbb{R}^{N \times n})}) | \sg u_k| >0$ for all $k \in \N$ because of~\eqref{eq:coercivcond}, also the sequence $(\sg u_k)_{k \in \N}$ is uniformly integrable. This together with the pointwise convergence~\eqref{eq:pointwise1} finishes the proof of the claim~\eqref{claim} and thus of the corollary.
\end{proof}
 
\section{Relaxation to $\bv$ and the Dual Problem}\label{sec:relaxedanddual}

The purpose of this section is to first recall the relaxed formulation of the minimisation problem~\eqref{eq:varprin}, namely the extension of the functional via semi-continuity to the space of functions of bounded variation, and the notion of generalised minimisers. Secondly, by means of convex conjugate functions in the sense of convex analysis, we introduce the dual problem associated to the (primal) minimisation problem~\eqref{eq:varprin} with an explicit description and then study its connection to the primal problem. In doing so, we shall adopt a more general viewpoint and hereafter let $F \colon \R^{N\times n}\to [0,\infty)$ a be convex, differentiable function that satisfies, for some constants $0<\nu\leq L<\infty$, the linear growth condition
\begin{align}\label{eq:generalisedlingrowth}
\nu|z| - L \leq F(z)\leq L(1+|z|)\qquad\text{for all } z\in\R^{N\times n}.
\end{align}
For a given map $T_{0} \in \lebe^\infty(\Omega;\R^{N\times n})$ we shall then study the variational problem
\begin{align}\label{eq:varprin_general}
\text{to minimise} \quad \mathcal{F}[w]\coloneqq \int_{\Omega} F(\nabla w)- T_{0} \cdot \nabla w \dif x \qquad \text{among all } w\in\sobo^{1,1}(\Omega;\R^{N}).
\end{align}
As for the radially symmetric case, we observe that \emph{if} a solution $u \in \sobo^{1,1}(\Omega;\R^N)$ to~\eqref{eq:varprin_general} exists, then it solves the associated Euler--Lagrange system
\begin{equation}
\label{eq:EL_general}
\int_\Omega \D_z F(\nabla u) \cdot \nabla \varphi \dif x = \int_\Omega T_{0} \cdot \nabla \varphi \dif x \qquad\text{for all } \varphi \in \sobo^{1,1}(\Omega;\R^N)
\end{equation}
and vice versa.

\subsection{Coerciveness}

As a modification of the coerciveness condition for radially symmetric integrands~\eqref{eq:coercivcond}, in this section we shall work with the condition
\begin{align}\label{eq:generalisedcoercivcond}
\essinf_{x \in \Omega} \min_{\xi \in \mathbb{S}^{N \times n -1}} \big\{ F^{\infty}(\xi) - T_0(x) \cdot \xi \big\} > 0
\end{align}
where the \emph{recession function} $F^{\infty}\colon\R^{N\times n}\to\R$ is given by
\begin{equation}
\label{eq:def_recession}
F^{\infty}(z) \coloneqq \lim_{t\searrow 0}tF\Big(\frac{z}{t}\Big) \qquad \text{for all } z\in\R^{N\times n}.
\end{equation}
We note that $F^\infty$ is strictly positive, finite-valued and convex, as a consequence of the linear growth condition and the convexity of~$F$, and hence, it attains its strictly positive minimum on $\mathbb{S}^{N \times n-1} = \{z \in \R^{N \times n} \colon |z|=1\}$. Also here the significance of condition~\eqref{eq:generalisedcoercivcond}, as previously for~\eqref{eq:coercivcond}, is to guarantee coerciveness of the functional~$\mathcal{F}$ in the following sense.

\begin{lemma}
\label{lemma:coerciveness_general}
Let $F \colon \R^{N\times n}\to [0,\infty)$ be a convex function satisfying~\eqref{eq:generalisedlingrowth} and let $T_{0}\in\lebe^{\infty}(\Omega;\R^{N\times n})$ verify~\eqref{eq:generalisedcoercivcond}. Then the functional $\mathcal{F}$ defined in~\eqref{eq:varprin_general} is coercive in the sense that if $(w_{k})_{k \in \N}$ is a sequence in $\sobo^{1,1}(\Omega;\R^{N})$ such that $\|w_{k}\|_{\sobo^{1,1}(\Omega;\R^{N})} \to \infty$ as $k\to\infty$ and each $w_{k}$ has vanishing mean value, then $\mathcal{F}[w_{k}]\to \infty$ as $k\to\infty$.
\end{lemma}

\begin{proof}
We initially observe that due to condition~\eqref{eq:generalisedcoercivcond} we may fix a number $\delta >0$ depending only on $F$ and $T_0$ such that
\begin{equation}
\label{eq_coerc_delta}
\essinf_{x \in \Omega} \min_{\xi \in \mathbb{S}^{N \times n -1}} \big\{ F^{\infty}(\xi) - T_0(x) \cdot \xi \big\} \geq 4 \delta
\end{equation}
is satisfied. We now consider an arbitrary function $w \in \sobo^{1,1}(\Omega;\R^{N})$ with $(w)_\Omega = 0$. In order to evaluate $\mathcal{F}[w]$, we decompose the domain of integration for some $\ell_0 \geq 1$ (to be determined later) as
\begin{align*}
\mathcal{F}[w] & = \int_{\Omega \cap \{|\nabla w|\leq \ell_0 \}} \big[ F(\nabla w)-  T_{0} \cdot \nabla w \big] \dif x + \int_{\Omega \cap \{ |\nabla w| > \ell_0 \}} \big[ F(\nabla w)- T_{0} \cdot \nabla w \big] \dif x \\
 &  \geq \int_{\Omega \cap \{|\nabla w|\leq \ell_0 \}} \big[ F(\nabla w)-  T_{0} \cdot \nabla w \big] \dif x \\
 & \quad +\int_{\Omega \cap \{|\nabla w| > \ell_0 \}}\bigg(\frac{1}{|\nabla w|} F\Big(|\nabla w|\frac{\nabla w}{|\nabla w|}\Big)- T_0 \cdot \frac{\nabla w}{|\nabla w|}- \delta \bigg) |\nabla w| \dif x\\
 & \quad + \delta \int_{\Omega \cap \{ |\nabla w| >\ell_0 \}} |\nabla w| \dif x =: \mathbf{I}+\mathbf{II}+\mathbf{III}.
\end{align*}
For the first term, we obtain via the growth condition~\eqref{eq:generalisedlingrowth}
\begin{align*}
|\mathbf{I}|\leq \big(C(1+\ell_0)+\ell_0\|T_{0}\|_{\lebe^{\infty}(\Omega;\R^{N\times n})}\big)\mathscr{L}^{n}(\Omega).
\end{align*}
We next show that the second term is non-negative, provided  that the level~$\ell_0$ is chosen suitably. To this end, we choose a finite number of points $(\xi_k)_{k \in \{1,\ldots,M\}}$ in~$\mathbb{S}^{N \times n-1}$ such that
\begin{equation*}
L_F \inf_{k \in \{1,\ldots,M\}} | \xi - \xi_k | \leq \delta \quad \text{for all } \xi \in \mathbb{S}^{N \times n-1},
\end{equation*}
where $L_F$ is a Lipschitz constant for both functions~$F$ and~$F^\infty$. Thus,~$M$ depends only on $n$, $N$, $\delta$ and~$F$. Taking into account that $\ell\mapsto F(\ell \xi )/\ell$ is monotonically increasing and converges to $F^{\infty}(\xi)$ as $\ell\nearrow\infty$ for each $\xi \in \mathbb{S}^{N \times n-1}$, we then determine $\ell_0 \geq 1$ such that
\begin{equation*}
\ell^{-1} F(\ell \xi_k )  \geq F^{\infty}(\xi_k) - \delta \qquad \text{for all } k \in \{1,\ldots,M\} \text{ and } \ell \geq \ell_0 \,.
\end{equation*}
Consequently, by the Lipschitz continuity of~$F$ and~$F^\infty$, we find
\begin{align*}
  \ell^{-1} F(\ell \xi) & \geq  \ell^{-1} \sup_{k \in \{1,\ldots,M\}} \big[ F(\ell \xi_k) - |F(\ell \xi) - F(\ell \xi_k)| \big] \\
    & \geq  \sup_{k \in \{1,\ldots,M\}} \big[  \ell^{-1}  F(\ell \xi_k) - L_F |\xi - \xi_k| \big] \\
    & \geq  \sup_{k \in \{1,\ldots,M\}}\big[  F^{\infty}(\xi) - \delta - 2L_F |\xi - \xi_k|  \big] \geq F^{\infty}(\xi) - 3 \delta
\end{align*}
for all $\xi \in \mathbb{S}^{N \times n-1}$ and $\ell \geq \ell_0$. Applying this inequality pointwisely with $\xi = \nabla w / |\nabla w|$ and keeping in mind the choice of~$\delta$ in~\eqref{eq_coerc_delta}, we thus arrive at $\mathbf{II} \geq 0$ as claimed. Finally, we observe
\begin{equation*}
\mathbf{III} \geq \delta \bigg[ \int_{\Omega}|\nabla w|\dif x -\ell_0 \mathscr{L}^{n}(\Omega)\bigg].
\end{equation*}
In conclusion, we have shown
\begin{equation*}
 \mathcal{F}[w] \geq \delta \int_{\Omega}|\nabla w|\dif x - C(F,T_0) \ell_0 \mathscr{L}^{n}(\Omega)
\end{equation*}
for all functions $w \in \sobo^{1,1}(\Omega;\R^{N})$ with $(w)_\Omega = 0$, and in combination with Poincar\'{e}'s inequality for $\sobo^{1,1}$-maps with vanishing mean value, this immediately implies the assertion of the lemma.
\end{proof}

\smallskip

\begin{remark}
\strut
\begin{enumerate}[font=\normalfont, label=(\roman{*}), ref=(\roman{*})]
 \item Relying on linear functions as in Examples~\ref{ex:nonexistence1} and~\ref{ex:nonexistence2} one shows optimality of condition~\eqref{eq:generalisedcoercivcond} concerning the existence of $\sobo^{1,1}$-minimisers for the Neumann problem for the functional~$\mathcal{F}$ \textup{(}and examples with unboundedness of~$\mathcal{F}$ from below when the expression in~\eqref{eq:generalisedcoercivcond} is strictly negative\textup{)}.
 \item However, since the minimisation problem~\eqref{eq:varprin_general} \textup{(}or~\eqref{eq:varprin}\textup{)} is formulated in terms of $\di T_0$ only, we indeed have coerciveness \textup{(}which then gives rise to existence results of generalised minimisers\textup{)} for all $T_{0}\in\lebe^{\infty}(\Omega;\R^{N\times n})$ such that there exists $\widetilde{T}_{0}\in\lebe^{\infty}(\Omega;\R^{N\times n})$ which verifies~\eqref{eq:generalisedcoercivcond} and
 \begin{equation*}
  \int_\Omega  T_{0} \cdot \nabla w \dif x = \int_\Omega \widetilde{T}_{0} \cdot \nabla w \dif x \qquad\text{for all } w \in \sobo^{1,1}(\Omega;\R^N).
 \end{equation*}
\end{enumerate}
\end{remark}

\subsection{Relaxation of the Primal Problem}
As mentioned in the introduction, the lack of weak compactness of bounded sets in the non-reflexive space $\sobo^{1,1}(\Omega;\R^{N})$ suggests the passage to a space that enjoys better compactness properties. The natural candidate for such a space is given by $\bv(\Omega;\R^{N})$, the space of functions of \emph{bounded variation}. We say that a measurable mapping $w \colon \Omega\to\R^{N}$ belongs to $\bv(\Omega;\R^{N})$ if and only if $w\in\lebe^{1}(\Omega;\R^{N})$ and its distributional gradient can be represented by a finite $\R^{N\times n}$-valued Radon measure on $\Omega$, in symbols $\D w\in\mathcal{M}(\Omega;\R^{N\times n})$. Let us note that by the Riesz representation theorem for Radon measures, the latter conditions amounts to requiring
\begin{align*}
|\D w|(\Omega) = \sup\left\{\int_{\Omega} w \cdot \di(\varphi) \dif x\colon\;\varphi\in\hold_{c}^{1}(\Omega;\R^{N\times n}),\;|\varphi|\leq 1\right\}<\infty,
\end{align*}
In this case, we denote by $\nabla w \mathscr{L}^{n}$ the absolutely continuous and by $\D^s w$ the singular part in the Lebesgue decomposition of $\D w$ with respect to the Lebesgue measure $\mathscr{L}^{n}$. However, let us emphasize that $\nabla w$ is simply the density of the absolutely continuous part of $\D w$, but in general, it is \emph{not} the gradient of a $\sobo^{1,1}(\Omega;\R^N)$-function.

The relevant notions of convergences in $\bv(\Omega;\R^N)$ are those of weak-$\ast$ and of strict convergence, both being weaker than norm convergence:

\begin{definition}
Let  $(w_{k})_{k \in \N}$ be a sequence in $\bv(\Omega;\R^{N})$ and $w \in \bv(\Omega;\R^{N})$. We say that $(w_k)_{k \in \N}$ \emph{converges weakly-$\ast$} to~$w$ in $\bv(\Omega;\R^{N})$, in symbols $w_{k} \wstar w$, if $(w_{k})_{k \in \N}$ converges strongly to~$w$ in $\lebe^{1}(\Omega;\R^{N})$ and if $(\D w_{k})_{k \in \N}$ converges to $\D w$ on~$\Omega$ in the weak-$\ast$-sense for Radon measures as $k\to\infty$, i.e.,
\begin{equation*}
 \lim_{k \to \infty} \int_\Omega \varphi \dif \D w_k = \int_\Omega \varphi \dif \D w \qquad \text{for all } \varphi \in \hold_0(\Omega).
\end{equation*}
We further say that $(w_k)_{k \in \N}$ \emph{converges strictly} to~$w$ in $\bv(\Omega;\R^{N})$ if $(w_{k})_{k \in \N}$ converges strongly to~$w$ in $\lebe^{1}(\Omega;\R^{N})$ and if the variations $|\D w_{k}|(\Omega)$ converge to $|\D w|(\Omega)$ as $k\to\infty$.
\end{definition}

Most importantly for us, we have the following characterization of weak-$\ast$-convergence that a sequence $(w_{k})_{k \in \N}$ converges weakly-$\ast$ in $\bv(\Omega;\R^{N})$ if and only if it is bounded in $\bv(\Omega;\R^{N})$ and strongly convergent in $\lebe^1(\Omega;\R^N)$. Moreover, the space $(\bv\cap\hold^{\infty})(\Omega;\R^{N})$ is dense in $\bv(\Omega;\R^{N})$ with respect to strict (and thus also with respect to weak-$\ast$) convergence. For this and further results on the space $\bv$ we refer the reader to the monographs~\cite{AMBFUSPAL99,EVANSGARIEPY}.

In what follows we consider $T_{0}\in\hold_b(\Omega;\R^{N\times n})$ and assume for the functional~$\mathcal{F}$ defined in~\eqref{eq:varprin_general} the mild coerciveness condition
\begin{equation}
\label{eq:generalisedcoercivcond_weak}
\essinf_{x \in \Omega} \min_{\xi \in \mathbb{S}^{N \times n -1}} \big\{ F^{\infty}(\xi) - T_0(x) \cdot \xi \big\} \geq 0
\end{equation}
 (i.e., in contrast to the previous coerciveness condition~\eqref{eq:generalisedcoercivcond}, also equality is allowed), which excludes~$\mathcal{F}$ to be unbounded from below. In this situation we extend~$\mathcal{F}$, which a priori is defined only on $\sobo^{1,1}(\Omega;\R^{N})$, by lower semicontinuity to the larger space $\bv(\Omega;\R^{N})$. The resulting relaxed functional is given by
\begin{equation*}
\overline{\mathcal{F}}[w] \coloneqq \inf \Big\{ \liminf_{k \to \infty} {\mathcal F}[w_k] \colon (w_k)_{k \in \N} \text{ in } \sobo^{1,1}(\Omega,\R^N) \text{ with } w_k \wstar w \text{ in } \bv(\Omega;\R^N) \Big\}
\end{equation*}
for $w \in\bv(\Omega;\R^{N}$). We now introduce the concept of generalised minimisers:

\begin{definition}
\label{def:bv_minimiser}
Let $F \colon \R^{N\times n}\to [0,\infty)$ be a convex function satisfying~\eqref{eq:generalisedlingrowth} and let $T_{0}\in\hold_b(\Omega;\R^{N\times n})$. We call a function $u \in \bv(\Omega;\R^N)$ \emph{generalised minimiser} of the functional~$\mathcal{F}$ if~$u$ is a minimiser of the relaxed functional~$\overline{\mathcal{F}}$ in $\bv(\Omega;\R^N)$, i.e.,
\begin{equation*}
 \overline{\mathcal{F}}[u] \leq \overline{\mathcal{F}}[w] \qquad \text{for all } w \in \bv(\Omega;\R^N).
\end{equation*}
\end{definition}

We next provide a representation formula for the relaxed functional~$\overline{\mathcal{F}}$ (with the classical approach as employed for the Dirichlet problem), prove that the original minimisation problem~\eqref{eq:varprin} and the minimisation of the relaxed functional $\overline{\mathcal F}$ in fact lead to the same value and we also justify the name ``generalised minimiser''.

\begin{proposition}
\label{prop:characterisation_gen_minimiser}
Let $F \colon \R^{N\times n}\to [0,\infty)$ be a convex function satisfying~\eqref{eq:generalisedlingrowth} and let $T_{0}\in\hold_b(\Omega;\R^{N\times n})$ verify~\eqref{eq:generalisedcoercivcond_weak}. Then we have the representation formula
\begin{equation}\label{eq:rep}
 \overline{\mathcal{F}}[w] = \int_{\Omega} F(\nabla w) \dif x + \int_{\Omega} F^{\infty}\Big(\frac{\dif \D^{s}w}{\dif |\D^{s}w|}\Big)\dif |\D^{s}w| - \int_\Omega T_0 \cdot \dif \D w
 \end{equation}
for all $w \in\bv(\Omega;\R^{N})$ with corresponding Lebesgue-Radon-Nikod\v{y}m decomposition $\D w=\nabla w\mathscr{L}^{n}\res\Omega+\D^{s}w$. Here, $F^{\infty}$ is the recession function defined in~\eqref{eq:def_recession}. Moreover, there holds
\begin{equation}
\label{eq:mainrelations}
\inf_{\bv(\Omega;\R^N)} \overline{\mathcal{F}} = \inf_{\sobo^{1,1}(\Omega;\R^{N})}\mathcal{F},
\end{equation}
and a function $u \in \bv(\Omega;\R^N)$ is a generalised minimiser of~$\mathcal{F}$ if and only if $u$ is the weak-$\ast$ limit of a minimising sequence $(u_k)_{k \in \N}$ for~$\mathcal{F}$ in $\sobo^{1,1}(\Omega;\R^N)$.
\end{proposition}

\begin{proof}
Let us denote by~${\mathcal G}[w]$ the right-hand side of~\eqref{eq:rep}. We initially observe from the lower semicontinuity and the continuity part of Reshetnyak's Theorem~\ref{lem:reshetnyak}, respectively, that we have
\begin{equation}
\label{eq_rep_proof_1}
 \mathcal{G}[w] \leq  \liminf_{k \to \infty} \mathcal{G}[w_k] = \liminf_{k \to \infty} \mathcal{F}[w_k] 
\end{equation}
for all sequences $(w_k)_{k \in \N}$ in $\sobo^{1,1}(\Omega,\R^N)$ with $ w_k \wstar w$ in $\bv(\Omega;\R^N)$ and 
\begin{equation}
\label{eq_rep_proof_2}
 \mathcal{G}[w] = \lim_{k \to \infty} \mathcal{G}[w_k] = \lim_{k \to \infty} \mathcal{F}[w_k]  
\end{equation}
for all sequences $(w_k)_{k \in \N}$ in $\sobo^{1,1}(\Omega,\R^N)$ with $ w_k \wstar w$ in $\bv(\Omega;\R^N)$ and $|(\mathscr{L}^{n}, \D w_k)|(\Omega) \to |(\mathscr{L}^{n}, \D w)|(\Omega)$. Here, we have used Remark~\ref{rem_Reshetnyak_application} to apply Reshetnyak to the functional~$\mathcal{G}$ and also the fact that~$\mathcal{G}$ and~$\mathcal{F}$ coincide on $\sobo^{1,1}(\Omega,\R^N)$.

We will first prove that $\overline{\mathcal{F}}[w]=\mathcal{G}[w]$ holds for every fixed $w \in \bv(\Omega;\R^{N})$. Noting that inequality~\eqref{eq_rep_proof_1} is valid for any sequence $(w_k)_{k \in \N}$ in $\sobo^{1,1}(\Omega,\R^N)$ such that $w_k \wstar w$ in $\bv(\Omega;\R^N)$ as $k \to \infty$, we may pass to the infimum of the right-hand side of~\eqref{eq_rep_proof_1} over these approximating sequences, and we find
\begin{equation*}
 \mathcal{G}[w] \leq \overline{\mathcal{F}}[w].
\end{equation*}
To obtain the reverse inequality, we choose (e.g.~by mollification of a trace-preserving extension of~$w$), a sequence $(w_{k})_{k \in \N}$ in $\sobo^{1,1}(\Omega;\R^{N})$ with $w_k \wstar w $ in $\bv(\Omega;\R^N)$ and with $|(\mathscr{L}^{n}, \D w_k)|(\Omega) \to |(\mathscr{L}^{n}, \D w)|(\Omega)$. Then, by identity~\eqref{eq_rep_proof_2}, we get
\begin{equation}
\label{eq_rep_proof_3}
 \mathcal{G}[w]  = \lim_{k \to \infty} \mathcal{F}[w_k] \geq \overline{\mathcal{F}}[w]
\end{equation}
which concludes the proof of the representation formula~\eqref{eq:rep}.

In order to demonstrate that the two infima in~\eqref{eq:mainrelations} coincide, we first notice from~\eqref{eq_rep_proof_3} that
\begin{equation*}
 \overline{\mathcal{F}}[w] = \mathcal{G}[w] = \lim_{k \to \infty} \mathcal{F}[w_{k}] \geq \inf_{\sobo^{1,1}(\Omega;\R^{N})} \mathcal{F}
\end{equation*}
for arbitrary $w \in\bv(\Omega;\R^{N})$ (and the sequence $(w_{k})_{k \in \N}$ with $w_k \wstar w $ in $\bv(\Omega;\R^N)$ and with $|(\mathscr{L}^{n}, \D w_k)|(\Omega) \to |(\mathscr{L}^{n}, \D w)|(\Omega)$ as above). Passing to the infimum of~$\overline{\mathcal{F}}$ over $w \in \bv(\Omega;\R^N)$ and keeping in mind that~$\mathcal{F}$ and~$\overline{\mathcal{F}}$ coincide on $\sobo^{1,1}(\Omega;\R^{N}) \subset \bv(\Omega;\R^N)$, we thus arrive at
\begin{equation*}
\inf_{\bv(\Omega;\R^N)} \overline{\mathcal{F}} \geq \inf_{\sobo^{1,1}(\Omega;\R^{N})} \mathcal{F} \geq \inf_{\bv(\Omega;\R^N)} \overline{\mathcal{F}},
\end{equation*}
and the claim~\eqref{eq:mainrelations} follows.

Finally, we prove the characterization of generalised minimisers. Given an arbitrary generalised minimiser $u \in \bv(\Omega;\R^N)$ of~$\mathcal{F}$, we see as above that $u$ is the weak-$\ast$ limit of a sequence $(u_k)_{k \in \N}$ in $\sobo^{1,1}(\Omega;\R^N)$ (and with $|(\mathscr{L}^{n}, \D u_k)|(\Omega) \to |(\mathscr{L}^{n}, \D u)|(\Omega)$). Thus, as a consequence of~\eqref{eq_rep_proof_2}, the fact that~$u$ minimises $\mathcal{G} = \overline{\mathcal{F}}$ in $\bv(\Omega;\R^N)$ and the identity~\eqref{eq:mainrelations}, we infer that $(u_k)_{k \in \N}$ in $\sobo^{1,1}(\Omega;\R^N)$ is indeed a minimising sequence for~$\mathcal{F}$ in $\sobo^{1,1}(\Omega;\R^N)$. For the reverse implication let $(u_k)_{k \in \N}$ be a minimising sequence for~$\mathcal{F}$ in $\sobo^{1,1}(\Omega;\R^N)$ that converges weakly-$\ast$ to a function $u \in \bv(\Omega;\R^N)$. Then, by~\eqref{eq_rep_proof_1} and once again identity~\eqref{eq:mainrelations}, we deduce that $u$ is indeed a minimiser of~$\mathcal{G} = \overline{\mathcal{F}}$ in $\bv(\Omega;\R^N)$, i.e.,~$u$ is a generalised
minimiser of~$\mathcal{F}$. This finishes the proof of the proposition.
\end{proof}

Concerning generalised minimisers of~$\mathcal{F}$, we next wish to continue the discussion of the coerciveness condition on~$T_0$, which was started in Example~\ref{ex:nonexistence1}, by showing that it remains an essential ingredient for a positive existence result:

\begin{example}[Example~\ref{ex:nonexistence1}, continued]\label{ex:nonexistence3}
In the situation of Example~\ref{ex:nonexistence1}, observe that $\overline{\mathcal{F}}$ for $\mathcal{F}=\mathfrak{F}$ is given by
\begin{equation*}
\overline{\mathcal{F}}[w] =\int_{-1}^{1} \big[ \sqrt{1+|w'|^{2}} -1  -w' \big] \dif x + \int_{(-1,1)} \bigg[ \Big\vert\frac{\dif \D^{s}w}{\dif |\D^{s}w|}\Big\vert-\frac{\dif \D^{s}w}{\dif |\D^{s}w|} \bigg] \dif |\D^{s}w| 
\end{equation*}
for $w \in \bv((-1,1))$, where now $\D w= w'\mathscr{L}^{1}\res(-1,1)+\D^{s}w$ is the Lebesgue decomposition of~$\D w$.
From Example~\ref{ex:nonexistence1} and identity~\eqref{eq:mainrelations} we deduce $\inf_{\bv((-1,1))}\overline{\mathcal{F}}=-2$. In this case, the minimising sequence $(u_k)_{k\in \N}$ with $u_k = kx$ for $k \in \N$ is not uniformly bounded in $\sobo^{1,1}((-1,1))$ \textup{(}and admits no subsequence converging weakly-$\ast$ in $\bv((-1,1))$\textup{)}. In fact, there exists no generalised minimiser of~${\mathcal F}$, i.e.,~a function $u \in \bv((-1,1))$ with $\overline{\mathcal{F}}[u]=2$. Otherwise, this would mean
\begin{align*}
\int_{-1}^{1} \big( \sqrt{1+|u'|^{2}}-u' \big) \dif x = - \int_{(-1,1)} \bigg[ \Big\vert\frac{\dif \D^{s}u}{\dif |\D^{s}u|}\Big\vert-\frac{\dif \D^{s}u}{\dif |\D^{s}u|} \bigg]\dif |\D^{s}u|
\end{align*}
Now, since the left-hand side is non-negative due to $\sqrt{1+|\cdot|^{2}}\geq |\cdot|$ and since the right-hand side is non-positive, both terms actually need to vanish in order to achieve equality. We thus conclude $\sqrt{1+|u'|^{2}}-u'\equiv 0$ a.e.~in $(-1,1)$ \textup{(}as before in Example~\ref{ex:nonexistence1}\textup{)}, which yields a contradiction and shows that such a function~$u$ cannot exist.
\end{example}

\begin{proposition}\label{prop:existence_relaxed}
Let $F \colon \R^{N\times n}\to [0,\infty)$ be a convex function satisfying~\eqref{eq:generalisedlingrowth} and let $T_{0}\in\hold_b(\Omega;\R^{N\times n})$ verify~\eqref{eq:generalisedcoercivcond}. Then there exists a generalised minimiser $u \in\bv(\Omega;\R^{N})$ of~$\mathcal{F}$.
\end{proposition}

\begin{proof}
Let $(u_k)_{k \in \N}$ be a minimising sequence for~${\mathcal F}$ in $\sobo^{1,1}(\Omega;\R^N)$. Since ${\mathcal F}$ depends only on the gradient variable, we may assume $(u_k)_\Omega = 0$ for each $k \in \N$. As a consequence of Lemma~\ref{lemma:coerciveness_general} and $\inf_{\sobo^{1,1}(\Omega;\R^N)} {\mathcal F} < \infty$, we obtain boundedness of $(u_k)_{k \in \N}$ in $\sobo^{1,1}(\Omega;\R^{N})$. By weak-$\ast$-compactness of $\bv(\Omega;\R^N)$ we thus find that $(u_k)_{k \in \N}$ converges weakly-$\ast$, up to the passage to a subsequence, to a function $u\in\bv(\Omega;\R^{N})$. We finally conclude that~$u$ is in fact a generalised minimiser of~$\mathcal{F}$, in view of the characterisation in Proposition~\ref{prop:characterisation_gen_minimiser}.
\end{proof}

We conclude this subsection with two remarks.

\begin{remark}[Possible non-uniqueness of generalised minimisers]
\label{rem_non_uniqueness}
Similarly as for the Dirichlet problem, generalised minimisers of~$\mathcal{F}$ in the Neumann problem can in principle be non-unique, due to the occurrence of the recession function~$F^\infty$, which is only convex, but not strictly convex. If we could show that~$\D^s u$ does in fact vanish for one generalised minimiser~$u$, then we would find a minimiser of the original Neumann problem~\eqref{eq:varprin_general}. Thus, the passage to the relaxed formulation could be avoided and furthermore, it is easy to see that if~$F$ is even strictly convex, every generalised minimiser of~$\mathcal{F}$ is in fact already in $\sobo^{1,1}(\Omega;\R^{N})$ and consequently a standard minimiser of~$\mathcal{F}$.

As we have shown in Section~\ref{sec:main}, this indeed happens if the integrand $F$ is of radial structure and the hypotheses of Theorem~\ref{thm:main} are satisfied. Moreover, it is not too difficult to show that it is also the case for not necessarily radially symmetric $\mu$-elliptic integrands $F\in\hold^{2}(\R^{N\times n})$ with bounded gradient and \emph{mild degeneration} $\mu\leq 3$, since one can here adapt the strategy of~\cite{BILDHAUER02} \textup{(}see also~\cite{BILDHAUERBUCH,BECSCH13}\textup{)} to show the existence of a locally bounded generalised minimiser of class $\sobo_{\locc}^{1,L\log L}(\Omega;\R^{N})\subset\sobo_{\locc}^{1,1}(\Omega;\R^{N})$.  
\end{remark}

\begin{remark}
If we are in the setting of Theorem~\ref{thm:main} with $T_{0}\in\hold_b(\Omega;\R^{N\times n})$ verifying~\eqref{eq:coercivcond}, then the function~$u$ from~\eqref{eq:BVconv} is, as a consequence of the characterization in Proposition~\ref{prop:characterisation_gen_minimiser}, a generalised minimiser of~$\mathcal{F}=\mathfrak{F}$. With the existence of the minimiser $v \in \sobo^{1,1}(\Omega;\R^{N})$ if~$\Omega$ is simply connected, the previous Remark~\ref{rem_non_uniqueness} thus provides an alternative proof of the fact $v-(v)_\Omega=u$ and then also weak convergence $u_k \rightharpoonup u$ in $\sobo^{1,1}(\Omega;\R^N)$ \textup{(}which improves to strong convergence, see Corollary~\ref{cor_strong_convergence}\textup{)}.
\end{remark}

\subsection{The Dual Problem}\label{sec:duality}

We next address a second approach to study the convex minimisation problem~\eqref{eq:varprin_general}, namely via the so-called \emph{dual problem} in the sense of convex duality (see e.g.~\cite{EKETEM99,DACOROGNA08} for extensive treatises on this subject). After the introduction of an associated dual functional, the dual problem consists in its maximisation over a suitable class in $\lebe^{\infty}(\Omega;\R^{N \times n})$, which then leads to the same value as for the original problem~\eqref{eq:varprin_general}. In contrast to this primal problem, there is no lack of compactness for the dual problem and a solution always exists, under the assumption~\eqref{eq:generalisedcoercivcond} on~$T_0$ and~$F$, and it is then important to link the solutions of the primal and of the dual problem (which is strongly influenced by regularity issues). The general approach follows essentially the one from the Dirichlet problem, but for the convenience of the reader we give a short overview on the results and strategy of
proof, since it is often simpler than for the corresponding result in the Dirichlet problem. Moreover, we address only regular integrands, and various extension could be given also for non-differentiable integrands, following the reference~\cite{BECSCH13b}.

We shall now start to collect some background facts regarding the dual problem associated to the Neumann problem~\eqref{eq:varprin_general}. For this purpose, we first introduce for an arbitrary function $g \colon \R^m \to \R \cup \{\infty\}$ the \emph{conjugate function} $g^{*}\colon \R^m \to\R\cup\{\infty\}$ by
\begin{align*}
g^{*}(z^{*})\coloneqq \sup_{z\in \R^m} \big\{  z^* \cdot z -g(z) \big\},\qquad \text{for all } z^{*}\in \R^m.
\end{align*}
By definition $g^*$ is convex and lower semi-continuous, and if~$g$ is of class~$\hold^1(\R^m)$, we further have the \emph{duality relation}
\begin{align}\label{eq:dualityrelation}
z^{*} =  \D_z g(z) \qquad \text{if and only if} \qquad g(z)+g^{*}(z^{*})= z^{*} \cdot z
\end{align}
for $z \in \R^m$ (while if $g$ is only convex, a similar relation holds for the subdifferential instead of the differential). Keeping in mind the particular situation of radially symmetric integrands as in Section~\ref{sec:main}, we notice

\begin{remark}[Radially Symmetric Integrands]
If~$g$ is radially symmetric, i.e.,~it is of the form $g(\cdot)=f(|\cdot|)$ for some function $f \colon \R \to \R$, then we have $g^{*}(\cdot)=f^{*}(|\cdot|)$. In fact, for each $z^* \in \R^m$, we have
\begin{equation*}
g^{*}(z^{*})  = \sup_{z\in \R^{m}} \big\{ z^{*} \cdot z - f(|z|) \big\} = \sup_{t\geq 0} \big\{ t |z^{*}| - f(t) \big\} = f^{*}(|z^{*}|).
\end{equation*}
\end{remark}

In order to set up the dual problem  to the Neumann problem~\eqref{eq:varprin_general} with convex integrand~$F$, let us first note, that for any $w \in\sobo^{1,1}(\Omega;\R^{N})$, we find, via $F(z) \geq z^{*} \cdot z - F^*(z^*)$ for all $z,z^* \in \R^{N \times n}$, the inequality
\begin{equation*}
\mathcal{F}[w] \geq \int_{\Omega} \big[ \chi \cdot \nabla w - F^{*}(T_{0}+\chi) \big] \dif x = - \int_{\Omega} F^{*}(T_{0}+\chi)\dif x
\end{equation*}
for every function $\chi \in \lebe_{\bot}^{\infty}(\Omega;\R^{N\times n})$, where we have set
\begin{equation*}
\lebe_{\bot}^{\infty}(\Omega;\R^{N\times n})\coloneqq \left\{\chi\in\lebe^{\infty}(\Omega;\R^{N\times n})\colon \int_{\Omega} \chi \cdot \nabla w \dif x = 0 \text{ for all }w\in\sobo^{1,1}(\Omega;\R^{N}) \right\}.
\end{equation*}

The \emph{dual problem} to~\eqref{eq:varprin_general} then is
\begin{equation}
\label{eq:varprin_dual}
\text{to maximise} \quad \mathcal{R}_{T_0}[\chi] \coloneqq - \int_{\Omega} F^{*}(T_{0}+\chi)\dif x \qquad \text{among all } \chi \in \lebe_{\bot}^{\infty}(\Omega;\R^{N\times n}),
\end{equation}
and by passing to the infimum among all $w \in \sobo^{1,1}(\Omega;\R^N)$ and to the supremum among all $\chi \in \lebe_{\bot}^{\infty}(\Omega;\R^{N\times n})$, we immediately obtain
\begin{equation}
\label{eq:infsup_1}
\inf_{\sobo^{1,1}(\Omega;\R^{N})}\mathcal{F} \geq \sup_{\lebe_{\bot}^{\infty}(\Omega;\R^{N\times n})}\mathcal{R}_{T_0}.
\end{equation}
which is the simpler inequality of the duality formula. The other inequality can either be settled by referring to the general theory of convex duality as outlined in the Appendix~\ref{sec:dualekelandtemam}, or by a suitable approximation procedure, for which the reader is referred to Remark~\ref{rem:approximationforthedualproblem}.

\begin{remark}
Let us make a comparison with the respective Dirichlet problem
\begin{equation*}
\text{to minimise} \quad w \mapsto \int_\Omega F(\nabla w) \dif x \qquad \text{among all } w \in u_0 + \sobo^{1,1}_0(\Omega;\R^N)
\end{equation*}
with prescribed boundary values $u_0 \in \sobo^{1,1}(\Omega;\R^N)$. In this case, the dual problem is
\begin{equation*}
\text{to maximise} \quad \chi \mapsto \int_\Omega \big[ \chi \cdot \nabla u_0 - F^*(\chi) \big] \dif x \qquad \text{among all } \chi \in \lebe_{\di}^{\infty}(\Omega;\R^{N\times n}),
\end{equation*}
where
\begin{equation*}
\lebe_{\di}^{\infty}(\Omega;\R^{N\times n})\coloneqq \left\{\chi\in\lebe^{\infty}(\Omega;\R^{N\times n})\colon \int_{\Omega} \chi \cdot \nabla w \dif x = 0\;\text{for all }w\in\sobo^{1,1}_0(\Omega;\R^{N}) \right\}.
\end{equation*}
In this sense, the fact that we allow for a larger set of competitor maps in the Neumann problem than for the Dirichlet problem is reflected by a smaller set of competitors in the respective dual problems.
\end{remark}

Concerning the connection between solutions of the primal and of the dual problem, let us first state the following simple observation.

\begin{lemma}\label{lem:connectionbetweenprimalanddualsol}
Consider a convex function $F \in \hold^1(\R^{N\times n})$ satisfying~\eqref{eq:generalisedlingrowth} and let $T_{0}\in\lebe^\infty(\Omega;\R^{N\times n})$. If $u\in\sobo^{1,1}(\Omega;\R^{N})$ is a minimiser of the primal problem~\eqref{eq:varprin_general}, then the unique maximiser of the dual problem~\eqref{eq:varprin_dual} is given by $\sigma=\D_z F(\nabla u)-T_{0}$.
\end{lemma}

\begin{proof}
We first note that $\sigma$ belongs to $\lebe_{\bot}^{\infty}(\Omega;\R^{N\times n})$, by boundedness of~$\sigma$ and the fact that~$u$ satisfies the Euler--Lagrange system~\eqref{eq:EL_general}, due to its minimality. The evaluation of~$\mathcal{R}_{T_0}$ in~$\sigma$, in combination with~\eqref{eq:dualityrelation} and once again~\eqref{eq:EL_general} (applied with $\varphi = u$), yields
\begin{align*}
\mathcal{R}_{T_0}[\sigma] & = -\int_{\Omega}F^{*}(T_{0}+\sigma)\dif x = - \int_{\Omega}F^{*}(\D_z F(\nabla u))\dif x\\
& =\int_{\Omega} \big[ F(\nabla u) - \D_z F(\nabla u) \cdot \nabla u \big] \dif x =  \int_{\Omega} \big[ F(\nabla u) - T_{0} \cdot \nabla u \big] \dif x = \inf_{\sobo^{1,1}(\Omega;\R^{N})}\mathcal{F},
\end{align*}
and~\eqref{eq:infsup_1} then shows that~$\sigma$ is a maximiser of~\eqref{eq:varprin_dual}. Moreover, if $\tilde{\sigma} \in \lebe_{\bot}^{\infty}(\Omega;\R^{N\times n})$ is any maximiser of the dual problem~\eqref{eq:varprin_dual}, then we deduce from the previous identity
\begin{equation*}
 -\int_{\Omega} F^{*}(T_{0}+\tilde{\sigma})  \dif x = \int_{\Omega} \big[ F(\nabla u) - T_{0} \cdot \nabla u \big] \dif x = \int_{\Omega} \big[ F(\nabla u) - (T_{0}+\tilde{\sigma}) \cdot \nabla u \big] \dif x.
\end{equation*}
Since by definition of the conjugate function $F^*$ we have
\begin{equation*}
-F^{*}(T_{0}+\tilde{\sigma}) \leq F(\nabla u) - (T_{0}+\tilde{\sigma}) \cdot \nabla u ,
\end{equation*}
we actually have equality $ F(\nabla u) + F^{*}(T_{0}+\tilde{\sigma}) = (T_{0} +\tilde{\sigma}) \cdot \nabla u$ a.e.~on $\Omega$. Thus, by~\eqref{eq:dualityrelation} we arrive at
\begin{equation*}
 T_0 + \tilde{\sigma}=\D_z F(\nabla u)  \qquad \text{a.e.~on $\Omega$}
\end{equation*}
which proves uniqueness of the maximiser of~\eqref{eq:varprin_dual}.
\end{proof}

As we have emphasized above, in general we do not know that a minimiser of~\eqref{eq:varprin_general} exists. However, we can still extract some information from minimising sequences (similarly as in \cite[Lemma~3.1]{BILFUC02a}).

\begin{lemma}\label{lem:connectionbetweenprimalanddualsol_2}
Consider a convex function $F \in \hold^1(\R^{N\times n})$ satisfying~\eqref{eq:generalisedlingrowth} and let $T_{0}\in\lebe^\infty(\Omega;\R^{N\times n})$ verify~\eqref{eq:generalisedcoercivcond} . If $(u_{k})_{k \in \N}$ is a minimising sequence of the primal problem~\eqref{eq:varprin_general}, then the sequence $(\D_z F(\nabla u_{k}) - T_0)_{k \in \N}$ converges weakly-$*$ in $\lebe^{\infty}(\Omega;\R^{N\times n})$ to the unique maximiser of the dual problem~\eqref{eq:varprin_dual}. 
\end{lemma}

\begin{proof}
Let~$\sigma$ be a weak-$*$ $\lebe^{\infty}$-cluster point of the sequence $(\D_z F(\nabla u_{k}) - T_0)_{k \in \N}$ and let $(\ep_k)_{k \in \N}$ be the null-sequence in $[0,\infty)$ defined by
\begin{equation*}
 \ep_k^2 \coloneqq \mathcal{F}[u_k] -  \inf_{\sobo^{1,1}(\Omega;\R^{N})}\mathcal{F}.
\end{equation*}
Here we follow the strategy of proof from~\cite[Section~5]{BECSCH13b}. In the first step, we want to pass to a sequence $(v_{k})_{k \in \N}$ in $\sobo^{1,1}(\Omega;\R^N)$, preserving the properties that
\begin{align}
\label{eq:v_k_minimsing}
& (v_k)_{k \in \N} \text{ is a minimising sequence of the primal problem~\eqref{eq:varprin_general},} \\
\label{eq:v_k_cluster_point}
& \text{$\sigma$ is a  weak-$*$ $\lebe^{\infty}$-cluster point of the sequence }(\D_z F(\nabla v_{k}) - T_0)_{k \in \N},
\end{align}
but with the additional benefit that we have
\begin{equation}
\label{eq:v_k_EL_limit}
\bigg| \int_\Omega \big(\D_z F(\nabla v_k) - T_0 \big) \cdot \nabla \varphi  \dif x \bigg| \leq \ep_k \| \nabla \varphi \|_{\lebe^1(\Omega;\R^{N \times n})} \qquad \text{for all } \varphi \in \sobo^{1,1}(\Omega;\R^{N})
\end{equation}
for each $k \in \N$. In fact, for each $k \in \N$ we may apply Ekeland's variational principle~\cite[Theorem~1.1]{EKELAND74} on the complete metric space $\{w \in \sobo^{1,1}(\Omega;\R^N) \colon (w)_\Omega=0\}$ (with metric induced by the norm $\| \nabla w \|_{\lebe^1(\Omega;\R^{N \times n})}$) to find a function $v_k \in \sobo^{1,1}(\Omega;\R^N)$ with average $(v_k)_{\Omega}=0$ such that
\begin{align*}
 & \mathcal{F}[v_k] \leq \mathcal{F}[u_k], \\
 & \| \nabla v_k - \nabla u_k \|_{\lebe^1(\Omega;\R^{N \times n})} \leq \ep_k, \\
 & \mathcal{F}[v_k] \leq \mathcal{F}[w] + \ep_k \| \nabla v_k - \nabla w \|_{\lebe^1(\Omega;\R^{N \times n})}  \quad \text{for all } w \in \sobo^{1,1}(\Omega;\R^N).
\end{align*}
As a consequence of the first inequality, we obtain~\eqref{eq:v_k_minimsing}, from the second inequality we infer the pointwise convergence $\nabla v_k - \nabla u_k \to 0$ a.e.~in~$\Omega$, for some subsequence, and thus~\eqref{eq:v_k_cluster_point}, and the third inequality actually means that $v_k$ is the minimiser of a perturbed functional, for which the first-order criterion for minimality then yields~\eqref{eq:v_k_EL_limit}.

In the second step, we now prove the claim of the lemma, with the sequence $(u_k)_{k \in \N}$ replaced by $(v_k)_{k \in \N}$ as constructed above. Via~\eqref{eq:v_k_EL_limit} we first observe that~$\sigma$ belongs to the space $\lebe_{\bot}^{\infty}(\Omega;\R^{N\times n})$ of admissible functions for the dual problem~\eqref{eq:varprin_dual}. By convexity and lower semi-continuity of $F^*$, the map $\chi \mapsto -\int_{\Omega} F^{*}(\chi)\dif x$ is upper semicontinuous with respect to weak-$\ast$-convergence in $\lebe^{\infty}(\Omega;\R^{N\times n})$. In combination with the duality relation~\eqref{eq:dualityrelation} we thus find (up to the passage to a suitable subsequence)
\begin{align*}
 \mathcal{R}_{T_0}[\sigma] = - \int_{\Omega} F^*(T_0 + \sigma) \dif x
  & \geq - \lim_{k \to \infty} \int_{\Omega} F^*(\D_z F(\nabla v_{k}))  \dif x \\
  & = \lim_{k \to \infty} \int_{\Omega} \big[ F(\nabla v_{k}) - \D_z F(\nabla v_{k}) \cdot \nabla v_{k} \big] \dif x \\
  & = \lim_{k \to \infty} \mathcal{F}[v_k] + \lim_{k \to \infty}  \int_{\Omega} \big( T_0 - \D_z F(\nabla v_{k}) \big) \cdot \nabla v_{k} \dif x \,.
\end{align*}
In view of~\eqref{eq:v_k_minimsing}, the first term on the right-hand side gives $\inf_{\sobo^{1,1}(\Omega;\R^{N})}\mathcal{F}$, while the second term vanishes, as a consequence of~\eqref{eq:v_k_EL_limit} (applied with $\varphi = v_k$) and the uniform boundedness of $(\nabla v_k)_{k \in \N}$ in $\lebe^1(\Omega;\R^{N \times n})$ in view of Proposition~\ref{lemma:coerciveness_general}. Thus, with~\eqref{eq:infsup_1}, we arrive at
\begin{equation}\label{eq:infsuprelationsconcluded}
 \mathcal{R}_{T_0}[\sigma] \geq \inf_{\sobo^{1,1}(\Omega;\R^{N})} \mathcal{F} \geq \sup_{\lebe_{\bot}^{\infty}(\Omega;\R^{N\times n})}\mathcal{R}_{T_0},
\end{equation}
hence, $\sigma$ is indeed a maximiser of the dual problem~\eqref{eq:varprin_dual}. Now, since~$F^*$ is essentially strictly convex (see~\cite[Theorem~26.3]{ROCKAFELLAR70}), the maximiser is in fact unique, and thus, the whole sequence $(\D_z F(\nabla u_{k}) - T_0)_{k \in \N}$ converges weakly-$*$ in $\lebe^{\infty}(\Omega;\R^{N\times n})$ to the dual solution~$\sigma$ as asserted in the lemma.
\end{proof}

\smallskip

\begin{remark}
\label{rem:approximationforthedualproblem}
\strut
\begin{enumerate}[font=\normalfont, label=(\roman{*}), ref=(\roman{*})]
 \item With~\eqref{eq:infsuprelationsconcluded} and the previously established inequality~\eqref{eq:infsup_1}, we have finished the proof of the duality correspondence
 \begin{equation*}
 \inf_{\sobo^{1,1}(\Omega;\R^{N})}\mathcal{F} = \sup_{\lebe_{\bot}^{\infty}(\Omega;\R^{N\times n})}\mathcal{R}_{T_0}.
 \end{equation*}
 \item Taking into account Proposition~\ref{lemma:coerciveness_general}, we obtain in particular the existence of a unique solution of the dual problem~\eqref{eq:varprin_dual}, under the assumptions of the previous lemma.
 \item In the above setting, with $T_{0}\in\hold_b(\Omega;\R^{N\times n})$ verifying~\eqref{eq:generalisedcoercivcond}, we have shown in Proposition~\ref{prop:existence_relaxed} the existence of a generalised minimiser $u \in \bv(\Omega;\R^N)$ to the primal problem~\eqref{eq:varprin_general}. In the situation, where a minimising sequence $(u_k)_{k \in \N}$ exists such that $u_k$ converges weakly-$\ast$ in $\bv(\Omega;\R^N)$ to $u$ and $\nabla u_k$ converges a.e.~in~$\Omega$ to the absolutely continuous part $\nabla u$ in the Lebesgue decomposition for $\D u$, we in fact find that $\sigma \coloneqq \D_z F(\nabla u) - T_0$ solves the dual problem~\eqref{eq:varprin_dual}.
\end{enumerate}
\end{remark}

We finish this section with a regularity statement for the solution of the dual problem, in the situation with radially symmetric integrands as in Theorem~\ref{thm:main}.

\begin{theorem}
\label{thm_regularity_dual}
Under the assumptions of Theorem~\ref{thm:main}, the dual problem~\eqref{eq:varprin_dual} with $F(\cdot) = f(|\cdot|)$ possesses a unique solution $\sigma\in \sobo_{\locc}^{1,2}(\Omega;\R^{N\times n})$ which is given by
\begin{equation}\label{eq:repdualsolution}
\sigma = f'(|\nabla u|)\frac{\nabla u}{|\nabla u|} - T_0 \qquad \mathscr{L}^{n}\text{-a.e.~in } \Omega,
\end{equation}
where $u \in\sobo^{1,1}(\Omega;\R^{N})$ is the minimiser of the primal problem from Theorem~\ref{thm:main}.
\end{theorem}

\begin{proof}
By Lemma~\ref{lem:connectionbetweenprimalanddualsol}, we obtain that $\sigma\coloneqq f'(|\nabla u|)\nabla u/|\nabla u|-T_{0} \in\lebe^{\infty}_{\bot}(\Omega;\R^{N\times n})$ is the unique solution of the dual problem. Thus, it only remains to verify the local $\sobo^{1,2}$-regularity of $\sigma$. To do so, we first recall from~\eqref{eq:unifbound1} and the pointwise convergence $\nabla u_k \to \nabla u$ $\mathscr{L}^{n}$-a.e.~in~$\Omega$ that $\sigma$ is a weak $\lebe^{2}$-cluster point of the sequence $(\sigma_{k})_{k \in \N}\coloneqq (\A_{k}(\nabla u_k)-T_{0})_{k \in \N}$, with $\A_k$ defined in~\eqref{eq:Aregularised} and with $u_k \in \sobo^{1,2}(\Omega;\R^{N})$ the minimiser of the functional~$\mathfrak{F}_{k}$ in~\eqref{eq:approx1}, for every $k \in \N$. Similarly as in the proof of Lemma~\ref{lem:sobolevregularity} we now exploit for each $k \in \N$ the fact that $\D_z A_k(\nabla u_k)$ is a positive definite, bilinear form, which is further bounded uniformly in view of ~\eqref{eq:growth_D_z_A}. By applying the Cauchy--Schwarz inequality similarly as in~\eqref{CS_sec_der} we then find, for each $s=1,\ldots,n$ and every compact set $K \subset \Omega$, the estimate
\begin{align*}
\lefteqn{ \int_{K} |\partial_{s}\sigma_{k}|^{2} \dif x = \int_{K} \D_z A_k(\nabla u_k) [ \partial_{s} \nabla u_{k} ,\partial_{s}\sigma_{k} ]  \dif x - \int_{K}  \partial_{s} T_{0} \cdot \partial_{s}\sigma_{k} \dif x } \\
 & \leq \bigg( \int_{K}  \D_z A_k(\nabla u_k) [ \partial_{s} \nabla u_{k}, \partial_{s} \nabla u_{k} ] \dif x \bigg)^{\frac{1}{2}} \bigg( \int_{K}  \D_z A_k(\nabla u_k) [ \partial_{s}\sigma_{k}, \partial_{s}\sigma_{k} ] \dif x \bigg)^{\frac{1}{2}} \\
 & \quad + \mathscr{L}^{n}(K)^{\frac{1}{2}} \|T_{0}\|_{\sobo^{1,\infty}(\Omega;\R^{N\times n})}\bigg( \int_{K} |\partial_{s}\sigma_{k}|^{2} \dif x \bigg)^{\frac{1}{2}} \\
 & \leq C \bigg(  \bigg( \int_{K} \D_z A_k(\nabla u_k) [ \partial_{s} \nabla u_{k}, \partial_{s} \nabla u_{k} ] \dif x \bigg)^{\frac{1}{2}}  +1\bigg) \bigg( \int_{K} |\partial_{s}\sigma_{k}|^{2} \dif x \bigg)^{\frac{1}{2}}
\end{align*}
with a constant~$C$ depending only on~$L$ and $\|T_{0}\|_{\sobo^{1,\infty}(\Omega;\R^{N\times n})}$. By an absorption argument and the local uniform estimate in Lemma~\ref{lem:sobolevregularity}, we hence deduce that $\sigma_k$ is even uniformly bounded in $\sobo^{1,2}(\Omega;\R^{N \times n})$, for each compact set $K \subset \Omega$. As a consequence, we deduce $\sigma \in \sobo_{\locc}^{1,2}(\Omega;\R^{N\times n})$ as claimed.
\end{proof}

\section{Appendix}
\label{sec:appendix}

We now collect some auxiliary and supplementary results that have occurred and been used in the main part of the paper.

\subsection{Reshetnyak-type Lower Semicontinuity Results}
We here state a result on the lower semicontinuity and continuity of convex variational integrals of linear growth due to Reshetnyak~\cite{RESHETNYAK68} (in the formulation of~\cite[Theorem~2.38 and Theorem~2.39]{AMBFUSPAL99} and~\cite[Theorem~2.4]{BECSCH13}) and then comment on its application in our setting.

\begin{theorem}[Reshetnyak (Lower Semi-)Continuity Theorem]\label{lem:reshetnyak}
Let $m\in\mathds{N}$, let $\Omega$ be a bounded, open subset of $\R^{n}$ and let $(\mu_{k})_{k \in \N}$ be a sequence in $\mathcal{M}(\Omega;\R^{m})$ that converges weakly-$\ast$ to some $\mu\in\mathcal{M}(\Omega;\R^{m})$. Moreover, assume that all $\mu,\mu_{1},\mu_{2},\ldots $ take values in some closed convex cone $K\subset\R^{m}$. Then we have the following statements:
\begin{enumerate}
\item \emph{(Lower Semicontinuity Part.)} If $G \colon \Omega \times K \to[0,\infty]$ is a lower semicontinuous function which is convex and $1$-homogeneous function in the second variable, then there holds
\begin{align*}
\int_{\Omega} G\Big( \cdot \,, \frac{\dif \mu}{\dif |\mu|}\Big)\dif |\mu| \leq \liminf_{k\to\infty}\int_{\Omega}G\Big(\cdot \,, \frac{\dif \mu_{k}}{\dif |\mu_{k}|}\Big)\dif |\mu_{k}|.
\end{align*}
\item \emph{(Continuity Part.)} If $G\colon \Omega \times K \to [0,\infty)$ is a continuous function which is $1$-homo\-geneous in the second variable  and if in addition $|\mu_{k}|(\Omega) \to |\mu|(\Omega)$ as $k \to \infty$, then there holds
\begin{align*}
\int_{\Omega}G\Big(\cdot\,, \frac{\dif \mu}{\dif |\mu|}\Big)\dif |\mu| = \lim_{k\to\infty}\int_{\Omega}G\Big(\cdot \,,\frac{\dif \mu_{k}}{\dif |\mu_{k}|}\Big)\dif |\mu_{k}|.
\end{align*}
\end{enumerate}
\end{theorem}

\begin{remark}
\label{rem_Reshetnyak_application}
In our setting, this result is applied as follows: given a convex function $F \colon \R^{N\times n}\to [0,\infty)$ of linear growth~\eqref{eq:generalisedlingrowth} and $T_{0}\in\hold_b(\Omega;\R^{N\times n})$ verifying the mild coerciveness condition~\eqref{eq:generalisedcoercivcond_weak}, we consider the half-space $K \coloneqq [0,\infty) \times \R^{N \times n}$, that is we choose $m=Nn+1$, and we define~$G$ on $\Omega \times K$ as the perspective integrand 
\begin{align*}
G(x,t,z) \coloneqq \begin{cases} t F(z/t) - T_0(x) \cdot z  &\quad \text{if } t>0,\\
F^{\infty}(z) - T_0(x) \cdot z  &\quad \text{if } t=0,
\end{cases}
\end{align*}
for all $x \in \Omega$, $t \in [0,\infty)$ and $z \in \R^{N \times n}$. In this situation it is easily checked that~$G$ takes values in~$[0,\infty)$ and that it is a continuous function which is convex and $1$-homogeneous in the second variable~$(t,z) \in K$. Hence, we have lower semicontinuity and continuity of~$G$ as stated in Theorem~\ref{lem:reshetnyak}, and with $\mu = (\mathscr{L}^{n},\D w)$ for an arbitrary function $w \in \bv(\Omega;\R^{N})$ we can rewrite the evaluation of~$G$ in terms of~$F$, the recession function $F^\infty$ and $T_0$ as
\begin{align*}
\lefteqn{ \int_{\Omega} G\Big(\cdot\,,\frac{\dif \, (\mathscr{L}^{n}, \D w)}{\dif |(\mathscr{L}^{n},\D w)| }\Big) \dif |(\mathscr{L}^{n},\D w)| } \\
  & = \int_{\Omega_r} G\Big(\cdot\,,\frac{(1,\nabla w)}{|(1,\nabla w)| }\Big) |(1,\nabla w)| \dif \mathscr{L}^{n}
  + \int_{\Omega_s} G\Big(\cdot\,,0, \frac{\dif \D^s w}{\dif |\D^s w| } \Big) \dif |\D^s w| \\
  & = \int_{\Omega} F(\nabla w) \dif x + \int_\Omega F^\infty \Big(\frac{\dif \D^s w}{\dif |\D^s w|} \Big) \dif |\D^s w| - \int_\Omega T_0 \cdot \dif \D w ,
\end{align*}
where by $\Omega_r, \Omega_s \subset \Omega$ we have denoted a disjoint decomposition of $\Omega$ with the property $\mathscr{L}^{n}(\Omega_s)=|\D^{s} w|(\Omega_r)=0$ and hence, for the densities we may use
\begin{equation*}
\frac{\dif \mathscr{L}^{n}}{ \dif |\D w|} = 0 \quad \text{and} \quad \frac{\dif \D w}{\dif |\D w|} = \frac{\dif \D^s w}{\dif |\D^s w|} \quad \text{on~}\Omega_s.
\end{equation*}
For the application of Theorem~\ref{lem:reshetnyak} we finally note that whenever $(w_k)_{k \in \N}$ is a sequence converging weakly-$*$ to some function~$w$ in $\bv(\Omega;\R^{N})$, then $(\mathscr{L}^{n}, \D w_k)$ converges weakly-$\ast$ to $(\mathscr{L}^{n}, \D w)$ in $\mathcal{M}(\Omega;\R^{N \times n +1})$. 
\end{remark}

\subsection{The Dual Problem in the Framework of Ekeland and Temam}\label{sec:dualekelandtemam}

In their treatise~\cite{EKETEM99}, Ekeland and Temam introduced a rather general framework of convex duality into which the Neumann problem on $\sobo^{1,1}(\Omega;\R^N)$ as described in our paper can be embedded in a natural way. Here we briefly discuss its relation to the setting of Section~\ref{sec:duality}.

In order to set up this framework, let $V,Y$ be two topological vector spaces with dual spaces $V^{*}$, $Y^{*}$ and suppose that a functional $\mathcal{F}\colon V\to \R\cup\{\infty\}$ can be written as
\begin{equation*}
\mathcal{F}[v]=J(v,\Lambda v) \qquad \text{for all } v \in V,
\end{equation*}
with a continuous, linear mapping $\Lambda\colon V\to Y$ and a convex function $J \colon V\times Y\to  \R\cup\{\infty\}$. When defining the convex conjugate function $J^* \colon V^* \times Y^* \to \R\cup\{\infty\}$ via
\begin{equation*}
 J^{*}(v^*,y^*) \coloneqq \sup_{v \in V, y \in Y} \big\{ \langle v^*,v \rangle_{V^* \times V} + \langle y^*,y \rangle_{Y^* \times Y} - J(v,y) \big\} \quad \text{for } v^* \in V^*, \, y^* \in Y^*,
\end{equation*}
we can introduce, following~\cite[Section~III.4]{EKETEM99}, the dual problem to the minimisation of~$\mathcal{F}$ over~$V$ in the sense of Ekeland and Temam as the problem
\begin{equation}
\label{dual_EKETEM99}
 \text{to maximise } -J^{*}(\Lambda^{*}y^{*},-y^{*}) \quad \text{among all } y^* \in Y^{*},
\end{equation}
where $\Lambda^{*}\colon Y^{*}\to V^{*}$ is the adjoint operator of~$\Lambda$ (with $\langle \Lambda^* y^*,v \rangle_{V^* \times V} = \langle y^*,\Lambda v \rangle_{Y^* \times Y}$ for all $v \in V$). Under the assumptions $\inf_{V}\mathcal{F} \in \R$ and that there exists $v_{0}\in V$ with $J(v_{0},\Lambda v_{0})<\infty$ and $p \mapsto J(v_{0},p)$ being continuous at $\Lambda v_{0}$, then by \cite[Theorem~III.4.1]{EKETEM99} there holds the duality correspondence
\begin{equation*}
\inf_{v\in V} \mathcal{F}[v]=\sup_{y^{*}\in Y^{*}} -J^{*}(\Lambda^{*}y^{*},-y^{*}).
\end{equation*}
We now specialize to the situation that the functional~$J_V$ splits into
\begin{equation*}
J(v,y)=J_V(v)+J_Y(y)  \qquad \text{for all } v \in V \text{ and } y \in Y,
\end{equation*}
with two convex functionals $J_V\colon V\to\R\cup\{\infty\}$ and $J_Y\colon Y\to\R\cup\{\infty\}$. The convex conjugate clearly preserves the splitting structure $J^{*}(v^*,y^*) = J_V^*(v^*) + J_Y^*(y^*)$ into the convex conjugates $J_V^* \colon V^* \to \R\cup\{\infty\}$ and~$J_Y^* \colon Y^* \to \R\cup\{\infty\}$ of~$J_V$ and~$J_Y$. Consequently, the dual problem here is to maximise
$-J_V^{*}(\Lambda^{*}y^{*})-J_Y^{*}(-y^{*})$ among all $y^* \in Y^{*}$.

In order to apply this abstract theory to the functional $\mathcal{F}$ in~\eqref{eq:varprin_general} (with~$F$ of linear growth~\eqref{eq:generalisedlingrowth} and with $T_{0}\in\lebe^{\infty}(\Omega;\R^{N\times n})$), we set $V\coloneqq \sobo^{1,1}(\Omega;\R^{N})$, $Y\coloneqq \lebe^{1}(\Omega;\R^{N\times n})$ and $\Lambda \coloneqq \nabla$ the weak gradient operator. We then define~$J$ in splitting form via the functionals $J_V \colon \sobo^{1,1}(\Omega;\R^{N}) \to \R$ and $J_Y \colon \lebe^{1}(\Omega;\R^{N\times n}) \to \R$ given as
\begin{equation*}
 J_V \equiv 0 \quad \text{and} \quad J_Y(y)\coloneqq \int_{\Omega} \big[ F(y) - T_0 \cdot y \big] \dif x \quad \text{for } y \in \lebe^{1}(\Omega;\R^{N\times n}).
\end{equation*}
For the identification of the dual problem~\eqref{dual_EKETEM99} with the integral formulation~\eqref{eq:varprin_dual}, let us first observe that we need to maximise among functions in $Y^*= \lebe^{\infty}(\Omega;\R^{N\times n})$. Moreover, since $J_V^{*}(\Lambda^{*}y^{*})=\infty$ whenever $\langle \Lambda^{*}y^{*},v \rangle_{V^* \times V} = \langle y^{*}, \Lambda v \rangle_{Y^* \times Y} \neq 0$ and $J_V^{*}(\Lambda^{*}y^{*})=0$ otherwise, it is sufficient to consider in the maximisation problem~\eqref{dual_EKETEM99} only $y^* \in Y^*$ with $\langle y^{*}, \Lambda v \rangle_{Y^* \times Y} = 0$ for all $v \in V$, which precisely amounts to requiring $y^{*}\in\lebe_{\bot}^{\infty}(\Omega;\R^{N})$ as used before in~\eqref{eq:varprin_dual}. Thus, it only remains to maximise $-J_Y^{*}(-y^{*})$ given by
\begin{align*}
 - J_Y^{*}(-y^{*}) & = - \sup_{y\in Y} \bigg\{ \langle - y^{*},y\rangle_{Y^{*}\times Y} - \int_{\Omega} \big[ F(y)- T_{0} \cdot y \big] \dif x\bigg\} \\
 &  = - \sup_{y \in Y} \bigg\{ \langle T_0 - y^* , y \rangle_{Y^* \times Y} - \int_\Omega F(y) \dif x \bigg\} = - \int_\Omega F^*(T_0 - y^*) \dif x,
\end{align*}
where we also used~\cite[Section~IV.1]{EKETEM99} to pass from the convex conjugate of the functional to the functional with convex conjugate integrand. In conclusion, this explains the choice of the space $\lebe_{\bot}^{\infty}(\Omega;\R^{N\times n})$ and the duality correspondence
\begin{equation}\label{eq:appendixconveq}
 \inf_{\sobo^{1,1}(\Omega;\R^{N})}\mathcal{F} = \sup_{y^* \in \lebe_{\bot}^{\infty}(\Omega;\R^{N\times n})}  - \int_\Omega F^*(T_0 - y^*) \dif x =  \sup_{\lebe_{\bot}^{\infty}(\Omega;\R^{N\times n})}\mathcal{R}_{T_0}.
 \end{equation}
from the perspective of convex analysis (and since $\lebe_{\bot}^{\infty}(\Omega;\R^{N\times n})$ is a linear space, the sign of~$y^*$ in this formula is irrelevant).

\subsection{Proof of Lemma~\ref{lem:consistencylemma}}\label{sec:consistencylemma}
We now demonstrate the consistency result, Lemma~\ref{lem:consistencylemma}.
\begin{proof}
In view of $u\in\sobo^{2,\infty}(\Omega;\R^{N})$, we may extend $\nabla u$ to a Lipschitz function on $\overline{\Omega}$. By minimality of~$u$, it~satisfies the Euler--Lagrange system~\eqref{eq:EL_system}, which, for $\varphi \in\sobo_{0}^{1,1}(\Omega;\R^{N})$, implies after the application of the integration by parts formula
\begin{equation*}
\int_{\Omega} \di\bigg(\frac{f'(|\nabla u|)\nabla u}{|\nabla u|} - T_0\bigg) \cdot \varphi \dif x = 0.
\end{equation*}
By arbitrariness of $\varphi \in\sobo_{0}^{1,1}(\Omega;\R^{N})$, we deduce~\eqref{eq:PDE} for almost every $x \in \Omega$ by use of the Du Bois--Reymond Lemma.

In order to prove the validity of the second identity~\eqref{eq:neumann}, we consider general test functions $\varphi\in\sobo^{1,1}(\Omega;\R^{N})$ in the Euler--Lagrange system~\eqref{eq:EL_system}. To this end, we localize at the boundary, via a family of function $(\eta_{\delta})_{\delta > 0}$ in $\hold^{2,1}(\R^n;[0,1])$ such that $\eta_{\delta}$ satisfies $\eta_{\delta}\equiv 1$ on $\partial\Omega$ and vanishes outside of $\{x\in\Omega \colon \dista(x,\partial\Omega)>\delta \}$, for each $\delta>0$ . Then, with the integration by parts formula and $\eta_{\delta}=1$ on $\partial \Omega$, we obtain
\begin{align*}
 0 & = \int_{\Omega} \bigg( \frac{f'(|\nabla u|)\nabla u}{|\nabla u|} - T_0 \bigg) \cdot \nabla (\varphi \eta_{\delta} ) \dif x \\
 & = \int_{\partial\Omega} \bigg(\frac{f'(|\nabla u|)\nabla u}{|\nabla u|}-T_{0}\bigg)\cdot \varphi \otimes \nu_{\partial\Omega} \dif\mathcal{H}^{n-1} -\int_{\Omega} \di\bigg(\frac{f'(|\nabla u|)\nabla u}{|\nabla u|}-T_{0}\bigg) \cdot \varphi \eta_{\delta}  \dif x.
\end{align*}
Then, by the $\hold^{2}$ regularity of~$f$ combined with $\nabla u, T_{0}\in\sobo^{1,\infty}(\Omega;\R^{N\times n})$ and by the convergence  $\eta_{\delta}(x) \to 0$ for all $x \in \Omega$ as $\delta\searrow 0$, Lebesgue's dominated convergence theorem shows that the second term on right-hand side of the previous equation vanishes in the limit $\delta\searrow 0$. Hence,~\eqref{eq:neumann} follows again by Du Bois--Reymond's Lemma on $\partial\Omega$.
\end{proof}

\subsection{Ubiquity of the $h$-monotonicity}\label{sec:hconvex}
We finally show that the $h$-monotonicity condition~\eqref{eq:hmonotone} is indeed satisfied for any strictly convex function   $f\in\hold^{2}(\R_{0}^{+})$ which satisfies $f(0) = f'(0)= 0$ and the linear growth condition~\eqref{eq:lingrowth} as assumed in Theorem~\ref{thm:main}. For this purpose, we compute for arbitrary $z,\zeta\in\R^{N\times n}$ with $z\neq 0$
\begin{align}\label{eq:formula_Hessian}
 \D_{zz} f(|z|) [\zeta,\zeta ] & =  \D_z \Big(\frac{f'(|z|)}{|z|}z \Big) [ \zeta, \zeta ] \nonumber\\
 & = \Big( f''(|z|) \frac{z \otimes z}{|z|^2} + \frac{f'(|z|)}{|z|} \frac{|z|^2 \id_{N \times n} - z \otimes z}{|z|^2} \Big) [ \zeta, \zeta ]\nonumber\\
 & = f''(|z|) \frac{(z\cdot\zeta)^{2}}{|z|^2} + \frac{f'(|z|)}{|z|} \frac{|z|^2 |\zeta|^2 - (z\cdot\zeta)^{2}}{|z|^2}\\ &  \geq \min\left\{f''(|z|),\frac{f'(|z|)}{|z|} \right\}|\zeta|^{2} \nonumber.
\end{align}
Here we have used that because of $(z\cdot \zeta)^{2}\leq |z|^{2}|\zeta|^{2}$, both terms on the penultimate line of the previous estimation are non-negative. We next define
\begin{align*}
 h(t) \coloneqq  \min\left\{f''(t),\frac{f'(t)}{t} \right\} \qquad \text{for } t>0.
\end{align*}
We observe that~$h$ is continuous on~$\R^+$, since $f\in\hold^{2}(\R_{0}^{+})$, and it can be continuously extended to~$\R_{0}^{+}$ by setting $h(0) = f''(0)$ (since $f'(0) =0$ implies $f'(t)/t \to f''(0)$ as $t \to 0$). Moreover,~$h$ is also strictly positive almost everywhere on~$\R_{0}^{+}$, since~$f'$ as the derivative of a strictly convex function is strictly monotonically increasing with $f'(0)=0$ and consequently we also have $f''>0$ almost everywhere on~$\R_{0}^{+}$. With~$h$ defined in this way, we can now continue to estimate~\eqref{eq:formula_Hessian} and conclude with $\D_{zz} f(|z|) [\zeta,\zeta] \geq  h(|z|) |\zeta|^{2}$ for all $z,\zeta\in\R^{N\times n}$, which is the claimed lower bound in~\eqref{eq:hmonotone}. To obtain also the upper bound in~\eqref{eq:hmonotone}, we use~\eqref{eq:formula_Hessian} to see that
\begin{equation}\label{eq:formula_Hessian_2}
 \D_{zz} f(|z|) [\zeta,\zeta ] \leq \max\left\{f''(|z|),\frac{f'(|z|)}{|z|} \right\}|\zeta|^{2}
\end{equation}
for arbitrary $z,\zeta\in\R^{N\times n}$ with $z\neq 0$. Since $f'$ is monotonically increasing with $f'(0)=0$, we can use~\eqref{eq:lingrowth} to get for each $\tau \in \R_{0}^{+}$
\begin{equation*}
0\le f'(\tau)\le \lim_{t\to \infty} f'(t) = \lim_{t\to \infty}\frac{f(t)}{t} \le L.
\end{equation*}
In addition, we know that $f'(t)/t$ is continuous for $t\in [0,\infty)$. Therefore, employing also the assumption~\eqref{eq:sec_der_bound}, the upper bound in~\eqref{eq:hmonotone} directly follows from~\eqref{eq:formula_Hessian_2}.

\end{document}